%% file: LMG1_arxiv.tex
\documentclass[10pt]{article}

\input{preamble.tex}

\usepackage[total={5.2in,9in},top=1in, left=1.65in]{geometry}
\usepackage{amsthm}

\newtheorem{theorem}{Theorem}
\newtheorem{proposition}[theorem]{Proposition}
\newtheorem{definition}[theorem]{Definition}
\newtheorem{lemma}[theorem]{Lemma}
\newtheorem{corollary}[theorem]{Corollary}

\theoremstyle{remark}
\newtheorem{remark}[theorem]{Remark}

\newtheorem{example}[theorem]{Example}

\begin{document}

\title{Integral models of certain PEL Shimura varieties with $\Gamma_1(p)$-type level structure}

\author{Richard Shadrach \footnote{Rice University, shadrach@rice.edu}}

\maketitle

\begin{abstract}
 We study $p$-adic integral models of certain PEL Shimura varieties with level subgroup at $p$ related to the $\Gamma_1(p)$-level subgroup in the case of modular curves. We will consider two cases: the case of Shimura varieties associated with unitary groups that split over an unramified extension of $\mathbb{Q}_p$ and the case of Siegel modular varieties. We construct local models, i.e. simpler schemes which are \'{e}tale locally isomorphic to the integral models. Our integral models are defined by a moduli scheme using the notion of an Oort-Tate generator of a group scheme. We use these local models to find a resolution of the integral model in the case of the Siegel modular variety of genus 2. The resolution is regular with special fiber a nonreduced divisor with normal crossings.
\end{abstract}

\input{LocalModelGamma1WithAcknowledgements.tex}

\bibliography{refs}{}
\bibliographystyle{alphanum}

\end{document}

%% file: preamble.tex
\usepackage{amsmath, amssymb}

\usepackage{mathrsfs}

\usepackage{enumerate}

\usepackage{graphicx}
\usepackage{tikz, tikz-cd}
\usepackage{paralist}
\usepackage{capt-of}

\usepackage{array}
\newcolumntype{L}[1]{>{\raggedright\let\newline\\\arraybackslash\hspace{0pt}}m{#1}}
\newcolumntype{C}[1]{>{\centering\let\newline\\\arraybackslash\hspace{0pt}}m{#1}}
\newcolumntype{R}[1]{>{\raggedleft\let\newline\\\arraybackslash\hspace{0pt}}m{#1}}

\tikzset{/tikz/commutative diagrams/every label/.append style={font=}}

\newcommand{\set}[1]{\left\{ #1 \right\}}
\newcommand{\generate}[1]{\langle #1 \rangle}
\DeclareMathOperator{\Aut}{Aut}

\newcommand{\Gal}{\text{Gal}}
\newcommand{\ST}[0]{\text{ST}}

\DeclareMathOperator{\Spec}{Spec}

\newcommand{\Proj}[0]{\text{Proj}}
\newcommand{\GL}[0]{\text{GL}}
\newcommand{\SL}[0]{\text{SL}}
\newcommand{\Gr}[0]{\text{Gr}}
\newcommand{\Sp}[0]{\text{Sp}}
\newcommand{\G}[0]{\mathbb{G}}
\newcommand{\GSp}[0]{\text{GSp}}
\newcommand{\R}[0]{\mathbb{R}}

\newcommand{\Q}[0]{\mathbb{Q}}
\newcommand{\Z}[0]{\mathbb{Z}}

\newcommand{\E}[0]{\mathbb{E}}
\newcommand{\C}[0]{\mathbb{C}}
\newcommand{\sC}[0]{\mathcal{C}}
\newcommand{\A}[0]{\mathbb{A}}
\newcommand{\F}[0]{\mathbb{F}}
\newcommand{\sF}[0]{\mathcal{F}}

\newcommand{\sL}[0]{\mathcal{L}}
\newcommand{\sO}[0]{\mathcal{O}}

\newcommand{\sU}[0]{\mathcal{U}}
\newcommand{\sQ}[0]{\mathcal{Q}}

\newcommand{\opp}[0]{\text{opp}}

\newcommand{\Tr}[0]{\text{Tr}}

\newcommand{\aff}[0]{\text{aff}}

\newcommand{\red}[0]{\text{red}}

\newcommand{\Lie}[0]{\text{Lie}}
\newcommand{\diag}[0]{\text{diag}}
\newcommand{\Mloc}[0]{M^\text{loc}}
\newcommand{\Bl}[0]{\text{Bl}}
\newcommand{\MlocGL}[0]{M^\text{loc}_\GL}
\newcommand{\MlocGSp}[0]{M^\text{loc}_\GSp}

\newcommand{\sG}[0]{\mathcal{G}}

\newcommand{\sA}[0]{\mathcal{A}}
\newcommand{\sZ}[0]{\mathcal{Z}}

\newcommand{\sW}[0]{\mathcal{W}}

\newcommand{\sE}[0]{\mathcal{E}}
\newcommand{\fund}[0]{\text{fund}}

\newcommand{\sAGL}[0]{\mathcal{A}^\GL}
\newcommand{\sAGSp}[0]{\mathcal{A}^\GSp}
\newcommand{\Sh}[0]{\text{Sh}}
\newcommand{\Adm}[0]{\text{Adm}}

\newcommand{\Iw}[0]{\text{Iw}}

\newcommand{\sheafhom}[0]{\mathcal{H}\textit{\kern -2pt om}}

\newcommand{\characteristic}[0]{\text{char}}

\newcommand{\bal}[0]{\text{bal}}
\newcommand{\norm}[0]{\text{norm}}

\newcommand{\dotsa}{\cdots}

\long\def\/*#1*/{}


%% file: LocalModelGamma1WithAcknowledgements.tex
\newcommand{\LambdaNotSplit}[0]{\Lambda}
\newcommand{\omegaNotSplit}[0]{\omega}

\newcommand{\LambdaSplitOne}[0]{e_{11}\Lambda}
\newcommand{\omegaSplitOne}[0]{\sF}

\newcommand{\LambdaSplitTwo}[0]{\LambdaSplitOne^*}
\newcommand{\omegaSplitTwo}[0]{\omegaSplitOne^*}

\newcommand{\p}{\mathfrak{p}}

 \section{Introduction}
  In the arithmetic study of Shimura varieties, one seeks to have a model of the Shimura variety over the ring of integers $\sO_E$, where $E$ is the completion of the reflex field $\E$ at some finite place $\p$. Denote by $\Sh_K(\G,X)$ the Shimura variety given by the Shimura datum $(\G,X)$ and choice of an open compact subgroup $K = \prod_\ell K_\ell \subset \G(\A_f)$, where $\A_f$ is the ring of finite rational ad\`{e}les. For Shimura varieties of PEL-type, which are moduli spaces of abelian varieties with certain (polarization, endomorphism, and level) structures, one can define such an integral model by proposing a moduli problem over $\sO_E$. The study of such models began with modular curves by Shimura and Deligne-Rapoport. More generally, Langlands, Kottwitz, Rapoport-Zink, Chai, and others studied these models for various types of PEL Shimura varieties. The reduction modulo $\p$ of these integral models is nonsingular if the factor $K_p \subset \G(\Q_p)$ is chosen to be ``hyperspecial'' for the rational prime $p$ lying under $\p$. However if the level subgroup $K_p$ is not hyperspecial, usually singularities occur. It is important to determine what kinds of singularities can occur, and this is expected to be influenced by the level subgroup $K_p$.

  In order to study the singularities of these integral models, significant progress has been made by finding ``local models''. These are schemes defined in simpler terms which control the singularities of the integral model. They first appeared in \cite{DP} for Hilbert modular varieties and in \cite{dJ} for Siegel modular varieties with Iwahori level subgroup. More generally in \cite{RZ}, Rapoport and Zink constructed local models for PEL Shimura varieties with parahoric level subgroup.

  In \cite{GortzGL} G\"{o}rtz showed that in the case of a Shimura variety of PEL-type associated with a unitary group which splits over an unramified extension of $\Q_p$, the Rapoport-Zink local models are flat with reduced special fiber. In \cite{GortzGSp}, the same is shown for the local models of Siegel modular varieties. On the other hand, Pappas has shown that these local models can fail to be flat in the case of a ramified extension \cite{P}. In \cite{PR1}, \cite{PR2}, and \cite{PR3}, Pappas and Rapoport give alternative definitions of the local models which are flat. More recently in \cite{PZ}, Pappas and Zhu have given a general group-theoretic definition of the local models which, for PEL cases, agrees with Rapoport-Zink local models in the unramified case and the alternative definitions in the ramified case.
  
  In this article, we will consider two particular types of Shimura varieties. First the unitary case, where the division algebra $B$ of dimension $n^2$ has center $F$, an imaginary quadratic extension of a totally real finite extension $F^+$ of $\Q$ which is unramified at $p$. We will make assumptions on $p$ so that the unitary group $\G$ in the Shimura datum splits over an unramified extension of $\Q_p$ as $\GL_n \times \G_m$. The second case is that of the Siegel modular varieties where the group in the Shimura datum is $\G = \GSp_{2n}$. We will refer to this as the symplectic case.
  
  We will consider various level subgroups $K = K_pK^p$. We will always choose $K^p$ sufficiently small so that the moduli problems considered below are representable by schemes. Our constructions are based off of the Rapoport-Zink integral and local models in the case where $K_p$ is an Iwahori subgroup of $\G(\Q_p)$. In all the situations we consider, $G = \G_{\Q_p}$ extends to a reductive group over $\Z_p$ and one can take an Iwahori subgroup as being the inverse image of a Borel subgroup of $G(\F_p)$ under the reduction $G(\Z_p) \rightarrow G(\F_p)$. Note that in the case of modular curves, these $K$ are called $\Gamma_0(p)$-level subgroups. However there is some ambiguity in calling such $K$ a $\Gamma_0(p)$-level subgroup for a PEL Shimura variety; indeed one may consider more generally a parahoric subgroup. As such, we will refer to these $K$ as an $\Iw_0(p)$-level subgroups.
  
  In the unitary case, we refer to $K$ as being an $\Iw_1(p)$-level subgroup when $K_p$ is the pro-unipotent radical of an Iwahori. In the presence of an $\Iw_0(p)$-level subgroup, the resulting integral model $\sAGL_0$ admits a moduli description in terms of chains of isogenies of abelian schemes $A_0 \rightarrow A_1 \rightarrow \dots \rightarrow A_n$ of dimension $n^2$ and degree $p^{2n}$, subject to certain additional conditions. Using Morita equivalence, we will associate with each $\ker(A_i \rightarrow A_{i+1})$ a finite flat group scheme of order $p$ denoted by $G_i$. The moduli problem defining the integral model with $\Iw_1(p)$-level, given in Section \ref{sect the integral model sAGL_1}, is then to also include a choice of Oort-Tate generator for each $G_i$ (see Section \ref{sect the integral model sAGL_1} for the notion of an Oort-Tate generator).
  
  In the symplectic case where $\G = \GSp_{2n}$, the moduli description of the integral model $\sAGSp_0$ associated with an $\Iw_0(p)$-level subgroup is again in terms of chains of isogenies of abelian schemes $A_0 \rightarrow A_1 \rightarrow \dots \rightarrow A_n$ now of dimension $n$ and degree $p$, subject to certain additional conditions. Here, we define $G_i = \ker(A_i \rightarrow A_{i+1})$. These are again finite flat group schemes of order $p$. As in the unitary case, we define the notion of an $\Iw_1(p)$-level subgroup so that the moduli description of the resulting integral model $\sAGSp_1$ is to also include a choice of Oort-Tate generator for each $G_i$. However such a subgroup is not given by the unipotent radical of an Iwahori, see Section \ref{subsect Iw_1(p)-level subgroup} for a group-theoretic description. Instead we refer to the unipotent radical of an Iwahori as being an $\Iw_1^\bal(p)$-level subgroup, and the resulting integral model $\sAGSp_{1,\bal}$ admits the following moduli description. One must not only choose an Oort-Tate generator for each $G_i$, but also for $G_i^* = \ker(\hat{A}_{i+1} \rightarrow \hat{A}_i)$, subject to a certain additional condition. Our notation reflects that of modular curves, where this has been called a Balanced $\Gamma_1(p)$-level subgroup \cite[Section 3.3]{KatzMazur}.

  In \cite{HR} Haines and Rapoport, interested in the local factor of the zeta function associated with the Shimura variety, constructed affine schemes which are \'{e}tale locally isomorphic to integral models of certain Shimura varieties where the level subgroup is the pro-unipotent radical of an Iwahori. This follows the older works of Pappas \cite{Pappas95} and Harris-Taylor \cite{HT}. Haines and Rapoport consider the case of a Shimura variety associated with a unitary group which splits locally at $p$ given by a division algebra $B$ defined over an imaginary quadratic extension of $\Q$. The cocharacter associated with the Shimura datum is assumed to be of ``Drinfeld type''.
  
  To study the singularities of the integral models that we will define, we also construct \'{e}tale local models. In order to describe our results in the unitary case, we begin by recalling the local model associated with $\Iw_0(p)$-level subgroup constructed in \cite{RZ}. In this introduction, we assume for simplicity that $F^+ = \Q$. We can choose an isomorphism $B_{\overline{\Q}_p} \cong M_n(\overline{\Q}_p) \times M_n(\overline{\Q}_p)$ so that the minuscule cocharacter $\mu : \G_{m,\overline{\Q}_p} \rightarrow G_{\overline{\Q}_p}$ is identified with 
   $$\mu(z) = \diag(1^{n-r}, (z^{-1})^r) \times \diag((z^{-1})^{n-r}, 1^r), \quad 1 \leq r \leq n-1.$$
  We will write concisely as $\mu = (0^{n-r}, (-1)^r)$. Then for a $\Z_p$-scheme $S$, an $S$-valued point of the local model $\Mloc_\GL$ of $\sAGL_0$ is determined by giving a diagram
  \begin{center}
   \begin{tikzcd}
    \sO_S^n \arrow{r}{\varphi_0} &
    \sO_S^n \arrow{r}{\varphi_1} &
    \dotsa \arrow{r}{\varphi_{n-2}} &
    \sO_S^n \arrow{r}{\varphi_{n-1}} &
    \sO_S^n \tikzset{commutative diagrams/column sep/normal=0cm} \\  
    \omegaSplitOne_0 \arrow{r} \arrow[hook]{u} &
    \omegaSplitOne_1 \arrow{r} \arrow[hook]{u} &
    \dotsa \arrow{r} &
    \omegaSplitOne_{n-1} \arrow{r} \arrow[hook]{u} &
    \omegaSplitOne_n \arrow[hook]{u} \tikzset{commutative diagrams/column sep/normal=0cm}&
   \end{tikzcd}
  \end{center}
  where $\varphi_i$ is given by the matrix $\diag(p^{i+1}, 1^{n-i-1})$ with respect to the standard basis, $\omegaSplitOne_i$ is an $\sO_S$-submodule of $\sO_S^n$, and Zariski locally on $S$, $\omegaSplitOne_i$ is a direct summand of $\sO_S^n$ of rank $r$. With $S = \MlocGL$, the determinants 
   $$\bigwedge^\text{top} \omegaSplitOne_i \rightarrow \bigwedge^\text{top} \omegaSplitOne_{i+1} \quad \text{and} \quad \bigwedge^\text{top} \sO_S^n/\omegaSplitOne_i \rightarrow \bigwedge^\text{top} \sO_S^n/\omegaSplitOne_{i+1}$$
  determine global sections $q_i$ and $q_i^*$ of the universal line bundles 
   $$\sQ_i = \left(\bigwedge^\text{top} \omegaSplitOne_i\right)^{-1} \otimes \bigwedge^\text{top} \omegaSplitOne_{i+1} \quad \text{and} \quad \sQ_i^* = \left(\bigwedge^\text{top} \sO_S^n/\omegaSplitOne_i\right)^{-1} \otimes \bigwedge^\text{top} \sO_S^n/\omegaSplitOne_{i+1}$$
  respectively.

  As shown in \cite{GortzGL}, the special fiber of the local model can be embedded into the affine flag variety for $\SL_n$ and identified with a disjoint union of Schubert cells. Let $U \subset \Mloc_\GL$ be an affine open neighborhood of the ``worst point'', i.e. the unique cell which consists of a single closed point, with $U$ sufficiently small so that each $\sQ_i^*$ is trivial. Choosing such a trivialization, we can then identify the sections $q_i^*$ with regular functions on $U$.
  \begin{theorem}
   The scheme 
    $$U_1 = \Spec_U\left(\sO[u_0, \dotsc, u_{n-1}]/(u_0^{p-1} - q_0^*, \dotsc, u_{n-1}^{p-1} - q_{n-1}^*)\right)$$
   is an \'{e}tale local model of $\sAGL_1$.
  \end{theorem}
  By \cite{GortzGL} we can take $U = \Spec(B_\GL)$ where 
   $$B_\GL = \Z_p[a_{jk}^i, i = 0, \dotsc, n-1,\  j = 1, \dotsc, n-r,\  k=1, \dotsc, r]/I$$
  and $I$ is the ideal generated by the entries of certain matrices. To make the above theorem completely explicit, we will compute the $q_i^*$ with respect to this presentation. They are given by the strikingly simple expression $q_i^* = a_{n-r,r}^{i+1}$ for $0 \leq i \leq n-1$ where the upper index is taken modulo $n$. As a result, the integral models with $\Iw_1(p)$-level structure are reasonably well-behaved and can be explicitly analyzed (see below for an example in low dimension).
  
  For the symplectic case, our construction of local models for $\sAGSp_1$ and $\sA_{1,\bal}^\GSp$ is similar to that of the unitary case. In particular, they are explicitly defined as well.

  It is also of interest to have certain resolutions of the integral model of the Shimura variety with ``nice'' singularities, for example one which is semi-stable or locally toroidal. This problem was considered in the case of $\Iw_0(p)$-level structure by Genestier \cite{Genestier}, Faltings \cite{Faltings}, de Jong \cite{dJTalk}, and G\"{o}rtz \cite{GortzComputing} among others. Using the explicitly defined local model, and in particular the rather simple expression for $q_i^*$, we will construct a resolution of $\sAGSp_1$ in the case $n=2$. A resolution of $\sA_{1,\bal}^\GSp$ can be constructed in a similar manner, see Section \ref{sect Iw_1^bal(p) resolution} for more details.
  \begin{theorem}
   \label{thm intro resolution}
   Let $\sA_1$ denote the moduli scheme for the Siegel modular variety of genus 2 with $\Iw_1(p)$-level structure. There is a regular scheme $\widetilde{\sA}_1$ with special fiber a nonreduced divisor with normal crossings\footnote{In this article, by a ``nonreduced divisor with normal crossings'' we mean a divisor $D$ such that in the completion of the local ring at any closed point, $D$ is given by $Z(f_1^{e_1} \cdots f_t^{e_t})$ where $\set{f_1, \dots, f_t}$ are part of a regular system of parameters and the integers $e_i$ are greater than zero.} that supports a birational morphism $\widetilde{\sA}_1 \rightarrow \sA_1$.
  \end{theorem}
  Moreover, we find the number of geometric irreducible components of $\widetilde{\sA}_1 \otimes \F_p$ and describe how they intersect using a ``dual complex'', see Theorem \ref{thm description of resolution of sA_1} for details. This resolution can be used to calculate the alternating trace of the Frobenius on the sheaf of nearby cycles in order to determine the local factor of the Hasse-Weil zeta function by \cite[Section 5]{ScholzeLanglands-KottwitzApproach}. This applies even when some of the multiplicities of the components are divisible by $p$ as we find in our case. We would like to work out the computation of the alternating trace in the future. Another use of this type of resolution can be found in the work of Emerton-Gee for unitary Shimura varieties \cite{Emerton-Gee}.
  
  Let us outline the construction of $\widetilde{\sA}_1$. We begin with a known semi-stable resolution $\widetilde{\sA}_0 \rightarrow \sA_0$ constructed by de Jong \cite{dJTalk}. This gives a modification (i.e. proper birational morphism) $\sA_1 \times_{\sA_0} \widetilde{\sA}_0 \rightarrow \sA_1$. The scheme $\sA_1 \times_{\sA_0} \widetilde{\sA}_0$ is not normal. Let $\sZ$ be the reduced closed subscheme of $\sA_0$ whose support is the locus of closed points where all of the corresponding kernels of the isogenies are infinitesimal. Also let $\sW$ be the unique irreducible component of the special fiber of $\sA_0$ where each kernel is generically isomorphic to $\mu_p$. Take the strict transform of $\sZ$ and $\sW$ with respect to the morphism $\widetilde{\sA}_0 \rightarrow \sA_0$. Denote by $\sZ'$ and $\sW'$ the reduced inverse image of these strict transforms with respect to the projection $\sA_1 \times_{\sA_0} \widetilde{\sA}_0 \rightarrow \widetilde{\sA}_0$. Consider the modification given by the blowup of $\sA_1 \times_{\sA_0} \widetilde{\sA}_0$ along $\sZ'$: 
   $$\Bl_{\sZ'}(\sA_1 \times_{\sA_0} \widetilde{\sA}_0) \rightarrow \sA_1 \times_{\sA_0} \widetilde{\sA}_0.$$
  We will see that $\Bl_{\sZ'}(\sA_1 \times_{\sA_0} \widetilde{\sA}_0)$ is normal. Let $\sW''$ denote the strict transform of $\sW'$ with respect to the modification $\Bl_{\sZ'}(\sA_1 \times_{\sA_0} \widetilde{\sA}_0) \rightarrow \sA_1 \times_{\sA_0} \widetilde{\sA}_0$. We arrive at $\widetilde{\sA}_1$ by first blowing up $\Bl_{\sZ'}(\sA_1 \times_{\sA_0} \widetilde{\sA}_0)$ along $\sW''$ and then blowing up each resulting modification along the strict transform of $\sW''$, stopping after a total of $p-2$ blowups. Carrying out the corresponding process on the local model, we will show by explicit computation that the resulting resolution of the local model is regular with special fiber a nonreduced divisor with normal crossings. It will then follow that $\widetilde{\sA}_1$ has these properties as well. By keeping track of how certain subschemes transform in each step of the above process, with much of this information coming from the explicit computation of the modifications of the local model, we will be able to find the number of geometric irreducible components of $\widetilde{\sA}_1 \otimes \F_p$ as well as describe how they intersect.
  
  We now describe the sections of the paper. In Section 2 we recall the construction of integral an local models of Shimura varieties in the unitary case with $\Iw_0(p)$-level subgroup as in \cite{RZ}. In Section 3 we construct integral and local models of Shimura varieties with $\Iw_1(p)$-level subgroup in the unitary case. In Section 4 we construct integral and local models of Shimura varieties with $\Iw_1(p)$- and $\Iw_1^\bal(p)$-level subgroup in the symplectic case. In Section 5 we construct the resolution mentioned in Theorem \ref{thm intro resolution} and in Section 6 we indicate how one can construct a similar resolution in the case of $\Iw_1^\bal(p)$-level.
  
  In closing, we mention that as this article was prepared, T. Haines and B. Stroh announced a similar construction of local models in order to prove the analogue of the Kottwitz nearby cycles conjecture. They relate their local models to ``enhanced'' affine flag varieties.
  
  I would like to thank G. Pappas for introducing me to this area of mathematics and for his invaluable support. I would also like to thank M. Rapoport for a useful conversation, T. Haines and B. Stroh for communicating their results, T. Haines for pointing out an error in a previous draft of this article, and U. G\"{o}rtz for providing the source for Figure \ref{gsp4 mu-permissible set} to which some modifications were made.
 \section{Integral and local model of $\sAGL_0$}
  \subsection{Integral model $\sAGL_0$}
   \label{Shimura Datum and Moduli Description}
   Let $F$ be an imaginary quadratic extension of a totally real finite extension $F^+/\mathbb{Q}$, $D$ a finite dimensional division algebra with center $F$, $*$ an involution on $D$ that induces the nontrivial element of $\Gal(F/F^+)$, and $h_0 : \C \rightarrow D_\R$ an $\R$-algebra homomorphism such that $h_0(z)^* = h_0(\overline{z})$ and the involution $x \rightarrow h_0(i)^{-1}x^*h_0(i)$ is positive. Let $p$ be an odd rational prime, and assume that $(p)$ is unramified in $F^+$ with each factor in $F^+$ splitting in $F$. We also assume that the division algebra $D$ splits over an unramified extension of $\Q_p$. We first consider the case where $F^+ = \Q$ and $D$ splits over $\Q_p$. We will return to the general case in Section \ref{The General Case}.
   
   The datum $(D, *, h_0)$ induces a PEL Shimura datum. We will briefly recall the crucial points, see \cite[Section 5]{H} for details. Set $B = D^\opp$ and $V = D$ viewed as a left $B$-module using right multiplications. Let $\G$ be the reductive group over $\Q$ defined by $\G(R) = \set{x \in D_R : x^*x \in R^\times}$ for a $\Q$-algebra $R$ and set $G = \G_{\Q_p}$. Fixing once and for all the embeddings $\C \hookleftarrow \overline{\Q} \hookrightarrow \overline{\Q}_p$, let $\mu : \G_{m,\overline{\Q}_p} \rightarrow G_{\overline{\Q}_p}$ be the minuscule cocharacter induced by $h_0$. We have that $h_0$ induces the decomposition $V_\C = V^+ \oplus V^-$ where $h_0(z \otimes 1)$ acts by $z$ on $V^+$ and by $\overline{z}$ on $V^-$. Choose an isomorphism $D_\C \cong M_n(\C) \times M_n(\C)$ such that $h_0(z \otimes 1) = \diag(\overline{z}^{n-r}, z^r) \times \diag(z^{n-r}, \overline{z}^r)$ for some $1 \leq r \leq n-1$. Then we have $\mu(z) = \diag(1^{n-r}, (z^{-1})^r) \times \diag((z^{-1})^{n-r},1^r)$ which we denote by $\mu = (0^{n-r}, (-1)^r)$. With $(p) = \p\p^*$ in $F$, we have $F_{\Q_p} = F_{\p} \times F_{\p^*}$ making $D_{\Q_p} = D_{\p} \times D_{\p^*}$. Note that $F_{\p} = F_{\p^*} = \Q_p$ and so $D_{\p}$ and $D_{\p^*}$ are $\Q_p$-algebras with $*$ inducing $D_{\p} \cong D_{\p^*}^\opp$. The assumption that $D$ splits over $\Q_p$ means $D_\p \cong M_n(\Q_p)$ and thus $G \cong \GL_{n,\Q_p} \times \G_{m,\Q_p}$. With $G$ split over $\Q_p$ we have that $E$, the reflex field at $p$, is $\Q_p$. Using the fixed embeddings $\C \hookleftarrow \overline{\Q} \hookrightarrow \overline{\Q}_p$, let $E' \subset \overline{\Q}_p$ be a finite extension of $E$ over which we have the decomposition $V_{E'} = V^+ \oplus V^-$.
   
   To define the integral model, we must specify a lattice chain $\sL$ and a maximal order $\sO_B$. Let $\xi \in D^\times$ be such that $\xi^* = -\xi$ and $\iota(x) = \xi x^* \xi^{-1}$ is a positive involution on $D$. Then $(\cdot, \cdot) : D \times D \rightarrow \Q$ defined by $(x,y) = \Tr_{D/\Q}(x\xi y^*)$ is a nondegenerate alternating pairing. Fix an isomorphism $D_{\Q_p} = D_\p \times D_{\p^*} \xrightarrow{\sim} M_n(\Q_p) \times M_n(\Q_p)$ so that $\iota$ goes to the standard involution $(X,Y) \rightarrow (Y^t, X^t)$. Set $(\chi^t, -\chi)$ to be the image of $\xi$ under this isomorphism. Let $\Lambda_i$ to be the lattice chain in $M_n(\Q_p)$ defined by $\Lambda_i = \diag((p^{-1})^i, 1^{n-i})M_n(\Z_p)$ for $0 \leq i \leq n-1$ and extend it periodically to all $i \in \Z$ by $\Lambda_{n+i} = p^{-1}\Lambda_i$. Likewise, $\Lambda_i^* = \chi^{-1} \diag(1^i, (p^{-1})^{n-i})M_n(\Z_p)$ for $0 \leq i \leq n-1$ and again extend it periodically. One may check that $(\Lambda_i \oplus \Lambda_i^*)^\perp = \Lambda_{-i} \oplus \Lambda_{-i}^*$ where
    $$(\Lambda_i \oplus \Lambda_i^*)^\perp = \set{x \in D_{\Q_p} : (x,y) \in \Z_p \text{ for all } y \in \Lambda_i \oplus \Lambda_i^*}$$
   and thus the lattice chain $\sL = (\Lambda_i \oplus \Lambda_i^*)_i$ is self-dual. Finally the maximal order $\sO_B \subset B$ is chosen so that under the identification $D_{\Q_p} \cong M_n(\Q_p) \times M_n(\Q_p)$ we have $\sO_B \otimes \Z_p = M_n^\opp(\Z_p) \times M_n^\opp(\Z_p)$. The level subgroup $K$ is taken to be a compact open subgroup of $\G(\A_f)$ of the form $K = K^pK_p$ where $K^p \subset \G(\A_f^p)$ is a sufficiently small compact open subgroup and $K_p = \Aut_{\sO_B}(\sL)$, the automorphism group of the polarized multichain $\sL$ \cite[3.23b]{RZ}.

   With this data, \cite[Definition 6.9]{RZ} gives an integral model of the Shimura variety $\sAGL_0$. It is the moduli scheme over $\Spec(\Z_p)$ representing the following functor. For any $\Z_p$-scheme $S$, $\sAGL_0(S) = \sAGL_{0,K^p}(S)$ is the set of tuples $(A_\bullet, \bar{\lambda}, \bar{\eta})$, up to isomorphism, where
   \begin{compactenum}[(i)]
    \item $A_\bullet$ is a chain $A_0 \xrightarrow{\alpha_0} A_1 \xrightarrow{\alpha_1} \dots \xrightarrow{\alpha_{n-1}} A_n$ of $n^2$-dimensional abelian schemes over $S$ determined up to prime-to-$p$ isogeny with each morphism $\alpha_i$ an isogeny of degree $p^{2n^2}$;
    \item each $A_i$ is equipped with an $\sO_B \otimes \Z_{(p)}$-action that commutes with the isogenies $\alpha_i$;
    \item there are ``periodicity isomorphisms" $\theta_p : A_{i+n} \xrightarrow{\sim} A_i$ such that for each $i$ the composition
      $$A_i \rightarrow A_{i+1} \rightarrow \dotsa \rightarrow A_{i+n} \xrightarrow{\theta_p} A_i$$
     is multiplication by $p$;
    \item the action of $\sO_B \otimes \Z_{(p)}$ satisfies the Kottwitz condition: for each $i$
      $$\textstyle \det_{\sO_S}(b; \Lie(A_i)) = \det_{E'}(b; V^+) \quad \text{for all } b \in \sO_B;$$
    \item $\bar{\lambda}$ is a $\Q$-homogeneous class of principal polarizations \cite[Definition 6.7]{RZ}; and
    \item $\bar{\eta}$ is a $K^p$-level structure \cite[Section 5]{Ko92}.
   \end{compactenum}
  \subsection{Local model diagram}
   We now describe the local model diagram
   $$\sAGL_0 \xleftarrow{\Phi} \widetilde{\sAGL_0} \xrightarrow{\Psi} \MlocGL$$
   constructed in \cite[Chapter 3]{RZ}. Let $(A_\bullet, \overline{\lambda}, \overline{\eta}) \in \sAGL_0(S)$. For an abelian scheme $A/S$ of relative dimension $n$, denote by $H_{dR}^1(A/S)^\vee$ the $\sO_S$-dual of the de Rham cohomology sheaf. It is a locally free $\sO_S$-module of rank $2n^2$ \cite[Section 2.5]{BBM} and the collection $(H_{dR}^1(A_i/S)^\vee)_i$ gives a polarized multichain of $\sO_B \otimes_{\Z_p} \sO_S$-modules of type $(\sL)$ \cite[3.23b]{RZ}. The scheme $\widetilde{\sAGL_0}$ represents the functor which associates with a $\Z_p$-scheme $S$ the set of tuples $(A_\bullet, \bar{\lambda}, \bar{\eta}, \set{\gamma_i})$, up to isomorphism, where $(A_\bullet, \bar{\lambda}, \bar{\eta}) \in \sAGL_0(S)$ and
    $$\gamma_i : H_{dR}^1(A_i/S)^\vee \rightarrow \Lambda_i \otimes_{\Z_p} \sO_S$$
   is an isomorphism of polarized multichains of $\sO_B \otimes_{\Z_p} \sO_S$-modules.  The morphism $\Phi : \widetilde{\sAGL_0} \rightarrow \sAGL_0$ is defined by sending $(A_\bullet, \bar{\lambda}, \bar{\eta}, \set{\gamma_i})$ to $(A_\bullet, \bar{\lambda}, \bar{\eta})$. Let $\sG$ be the smooth affine group scheme where $\sG(S) = \Aut(\set{\Lambda_i \otimes \sO_S})$ for a $\Z_p$-scheme $S$ \cite[Theorem 3.11]{RZ}. Then $\Phi$ is a smooth $\sG$-torsor \cite[Theorem 2.2]{P}. We now recall the definition of the local model.
   \begin{definition}\label{def RZ local model}
   \cite[Definition 3.27]{RZ}
   With $S$ a $\Z_p$-scheme, an $S$-valued point of $\MlocGL$ is given by the following data.
   \begin{compactenum}[(i)]
    \item A functor from the category $\sL$ to the category of $\sO_B \otimes_{\Z_p} \sO_S$-modules on $S$
     $$\Lambda \rightarrow \omega_\Lambda, \quad \Lambda \in \sL.$$
    \item A morphism of functors $\psi_\Lambda : \omega_\Lambda \rightarrow \Lambda \otimes_{\Z_p} \sO_S.$
   \end{compactenum}
   We require that the following conditions are satisfied.
   \begin{compactenum}[(i)]
    \item The morphisms $\psi_\Lambda$ are injective.
    \item The quotients $t_\Lambda := (\Lambda \otimes_{\Z_p} \sO_S) / \psi(\omega_\Lambda)$ are locally free $\sO_S$-modules of finite rank.  For the action of $\sO_B$ on $t_\Lambda$, we have the Kottwitz condition
     $$\textstyle\det_{\sO_S}(b; t_\Lambda) = \det_{E'}(b; V^+) \text{ for all } b \in \sO_B.$$
    \item For each $\Lambda \in \sL$, $\psi_\Lambda(\omega_\Lambda)^\text{perp} = \psi_{\Lambda^\perp}(\omega_{\Lambda^\perp})$.
   \end{compactenum}
  \end{definition}
   The definition of $\MlocGL$ is often rephrased so that $\omega_\Lambda$ is a subsheaf of $\Lambda \otimes_S \sO_S$ and the morphism $\psi_\Lambda$ is the inclusion. Note that $\sG$ acts on $\MlocGL$ by acting on $\psi_\Lambda$ through its natural action on $\Lambda \otimes \sO_S$. Set $\omegaNotSplit_{A_i} = \Lie(A_i)^\vee$, a locally free sheaf of rank $n^2$ on $S$. We have the Hodge filtration \cite[Prop. 5.1.10]{BBM}
    $$0 \rightarrow \omegaNotSplit_{\hat{A}_i} \rightarrow H_{dR}^1(A_i/S)^\vee \rightarrow \Lie(A_i) \rightarrow 0$$
   where $\hat{A}_i$ denotes the dual abelian scheme of $A_i$. The morphism $\Psi : \widetilde{\sAGL_0} \rightarrow \MlocGL$ is given by associating with each point $(A_\bullet, \bar{\lambda}, \bar{\eta}, \set{\gamma_i}) \in \widetilde{\sAGL_0}(S)$ the collection of injective morphisms $\omega_{\hat{A}_i} \rightarrow H_{dR}^1(A_i/S)^\vee \xrightarrow{\gamma_i} \Lambda_i \otimes_{\Z_p} \sO_S$.
   \begin{theorem}
    \cite[Chapter 3]{RZ} (cf.\ \cite[Section 2]{P} and \cite[Section 6]{H})
    \label{thm local model diagram}
    The morphism $\Phi$ is a $\sG$-torsor, $\Psi$ is a smooth $\sG$-equivariant morphism, and there is an \'{e}tale cover $\varphi : V \rightarrow \sAGL_0$ with a section $\sigma : V \rightarrow \widetilde{\sAGL_0}$ of $\Phi$ such that $\Psi \circ \sigma$ is \'{e}tale. That is, $\sAGL_0 \xleftarrow{\Phi} \widetilde{\sAGL_0} \xrightarrow{\Psi} \MlocGL$ forms a local model diagram in the terminology of \cite[Section 1.2]{LocalModelSurvey}.
   \end{theorem}
   
   \begin{remark}\label{remark corresponding closed point}
    Given a closed point $x$ of $\sAGL_0$, we will say that a closed point $y$ of $\MlocGL$ corresponds to $x$ if there exists $\varphi$ and $\sigma$ as in the above theorem, along with a closed point $v$ of $V$ such that $\varphi(v) = x$ and $\Psi \circ \sigma(v) = y$.
   \end{remark}

   The next proposition follows from the identification $\sO_B \otimes \Z_p = M_n^\opp(\Z_p) \times M_n^\opp(\Z_p)$, the Morita equivalence, and the Kottwitz condition. Let $e_{11} \in M_n^\opp(\Z_p) \times M_n^\opp(\Z_p)$ denote the idempotent of the first factor and write $\Lambda_{i,S} = \Lambda_i \otimes_{\Z_p} \sO_S$.
   
   \begin{proposition}\cite[Section 6.3.3]{H}\label{prop reduced description of MlocGL}
    Giving an $S$-valued point of $\MlocGL$ is equivalent to giving a commutative diagram
    \begin{center}
     \begin{tikzcd}
      e_{11}\Lambda_{0,S} \arrow{r}{\varphi_0} &
      e_{11}\Lambda_{1,S} \arrow{r}{\varphi_1} &
      \dotsa \arrow{r}{\varphi_{n-1}} &
      e_{11}\Lambda_{n,S} \\  
      \omegaSplitOne_0 \arrow{r} \arrow[hook]{u} &
      \omegaSplitOne_1 \arrow{r} \arrow[hook]{u} &
      \dotsa \arrow{r} &
      \omegaSplitOne_{n} \arrow[hook]{u}
     \end{tikzcd}
    \end{center}
    where $\varphi_i$ is the morphism induced from the inclusion $\Lambda_i \subset \Lambda_{i+1}$ and $\sF_i$ is a locally free $\sO_S$-module which is Zariski locally a direct summand of $\LambdaSplitOne_{i,S}$ of rank $r$.
   \end{proposition}
   In the next section we will use that for an $S$-valued point $\set{\omega_i \hookrightarrow \Lambda_{i,S}}$ of $\MlocGL$, the diagram above is given by setting $\sF_i = e_{11}\omega_i$.

 \section{The integral and local model of $\sAGL_1$}
  Throughout this section, $S$ is a $\Z_p$-scheme and $k$ is an algebraically closed field. 

  \subsection{The group schemes $G_i$}
   \label{sect the group schemes G_i}
   In order to define $\sAGL_1$, we will first associate with an $S$-valued point of $\sAGL_0$ a collection of finite flat group schemes $\set{G_i}_{i=0}^{n-1}$ each of order $p$. Then given a geometric point $x : \Spec(k) \rightarrow \sAGL_0$, we will determine the isomorphism type of the group schemes $G_i$ from the data given by a corresponding geometric point $y : \Spec(k) \rightarrow \MlocGL$ in the sense of Remark \ref{remark corresponding closed point}.
   
   Let $(A_\bullet, \overline{\lambda}, \overline{\eta}) \in \sAGL_0(S)$. For each $i \in \mathbb{Z}$, consider the $p$-divisible group $A_i(p^\infty) = \varinjlim_n A_i[p^n]$. Note that $H_i := \ker(A_i \rightarrow A_{i+1})$ is contained in $A_i(p^\infty)$. With $e_{11} \in M_n^\opp(\Z_p) \times M_n^\opp(\Z_p)$ the idempotent of the first factor, the action of $\sO_B \otimes \Z_p$ gives a chain 
    $$e_{11}A_0(p^\infty) \rightarrow e_{11}A_1(p^\infty) \rightarrow \dotsb \rightarrow e_{11}A_{n-1}(p^\infty) \rightarrow e_{11}A_n(p^\infty)$$
   of isogenies of degree $p$ whose composition with $e_{11}A_n(p^\infty) \xrightarrow{\sim} e_{11}A_0(p^\infty)$ induced by the periodicity isomorphism is multiplication by $p$. For $0 \leq i \leq n-1$, we define
    $$G_i = \ker\big(e_{11}A_i(p^\infty) \rightarrow e_{11}A_{i+1}(p^\infty)\big).$$
   Each $G_i$ is a finite flat group scheme of order $p$.
   \begin{definition}
    For a group scheme $G/S$, $\omega_G = \omega_{G/S}$ is the sheaf on $S$ given by $\varepsilon^*(\Omega_{G/S}^1)$ where $\varepsilon : S \rightarrow G$ is the identity section.
   \end{definition}
   
   When $S$ is affine and $G/S$ is a finite flat group scheme, we denote by $G^*$ the Cartier dual of $G$.
   
   \begin{proposition}\label{prop invariant differentals and quotients}
    Let $x : \Spec(k) \rightarrow \sAGL_0$ be a geometric point and let $y : \Spec(k) \rightarrow \MlocGL$ be a corresponding geometric point in the sense of Remark \ref{remark corresponding closed point}. Let $\set{\omega_i \hookrightarrow \Lambda_{i,S}}$ be the data induced by $y$ as in Definition \ref{def RZ local model} and set $\sF_i = e_{11}\omega_i$. With
     $$H_i = \ker(A_i \rightarrow A_{i+1}) \quad \text{and} \quad G_i = \ker\big(e_{11}A_i(p^\infty) \rightarrow e_{11}A_{i+1}(p^\infty)\big)$$
    we have
    \begin{compactenum}[(i)]
     \item $\dim_k \omega_{H_i^*} = \dim_k(\omegaNotSplit_{i+1}/\varphi_i(\omegaNotSplit_{i}))$;
     \item $\dim_k \omega_{H_i} = \dim_k \LambdaNotSplit_{i+1,S} / \left(\varphi_i(\LambdaNotSplit_{i,S}) + \omegaNotSplit_{i+1}\right)$;
     \item $\dim_k \omega_{G_i^*}  = \dim_k(\omegaSplitOne_{i+1}/\varphi_i(\omegaSplitOne_{i}))$; and
     \item $\dim_k \omega_{G_i} = \dim_k \LambdaSplitOne_{i+1,S} / \left(\varphi_i(\LambdaSplitOne_{i,S}) + \omegaSplitOne_{i+1}\right)$.
    \end{compactenum}
   \end{proposition}
   \begin{proof}
    To show the equalities in (i) and (ii), start with the standard exact sequence of K\"{a}hler differentials induced by the morphisms $\hat{A}_{i+1} \rightarrow \hat{A}_i \rightarrow S$ and $A_i \rightarrow A_{i+1} \rightarrow S$ respectively. Pull these sequences of sheaves back to $S$ by the appropriate identity section. The statements follow from the isomorphisms $\omega_{H_i^*} \cong \varepsilon_{\hat{A}_{i+1}}^*(\Omega_{\hat{A}_{i+1}/\hat{A}_i}^1)$ and $\omega_{H_i} \cong \varepsilon_{A_i}^*(\Omega_{A_i/A_{i+1}}^1)$. Now (iii) and (iv) follow from the functorality of the decompositions.
   \end{proof}

   For a general $S$-valued point $y : S \rightarrow \MlocGL$, the maps $\varphi_i : \omegaSplitOne_i \rightarrow \omegaSplitOne_{i+1}$ and $\varphi_i^* : \LambdaSplitOne_{i,S}/\omegaSplitOne_i \rightarrow \LambdaSplitOne_{i+1,S}/\omegaSplitOne_{i+1}$ induce global sections $q_i$ and $q_i^*$ of the line bundles $\sQ_i$ and $\sQ_i^*$ which are defined as
    $$\left(\bigwedge^\text{top} \omegaSplitOne_i\right)^{-1} \otimes \bigwedge^\text{top} \omegaSplitOne_{i+1} \quad \text{and} \quad \left(\bigwedge^\text{top}\LambdaSplitOne_{i,S}/\omegaSplitOne_i\right)^{-1} \otimes \bigwedge^\text{top} \LambdaSplitOne_{i+1,S}/\omegaSplitOne_{i+1}$$
   respectively. In the case $S = \Spec(k)$, $q_i$ and $q_i^*$ are the determinants of the corresponding linear maps. Note that $\Lambda_{i,S} \rightarrow \Lambda_{i+1,S}$ carries $\omegaSplitOne_i$ into $\omegaSplitOne_{i+1}$, so the determinant of $\Lambda_{i,S} \rightarrow \Lambda_{i+1,S}$ is the product of the determinants of $\omegaSplitOne_i \rightarrow \omegaSplitOne_{i+1}$ and $\LambdaSplitOne_{i,S}/\omegaSplitOne_i \rightarrow \LambdaSplitOne_{i+1,S}/\omegaSplitOne_{i+1}$. Thus $q_i \otimes q_i^* = p$.
   \begin{proposition}
    With the same hypotheses as Proposition \ref{prop invariant differentals and quotients}, denote by $q_i$ and $q_i^*$ the sections induced by $y$. Then $q_i = 0$ if and only if $\dim_k \omega_{G_i^*} = 1$, and $q_i^* = 0$ if and only if $\dim_k \omega_{G_i} = 1$.
   \end{proposition}
   \begin{proof}
    This follows immediately from Proposition \ref{prop invariant differentals and quotients} as $q_i \neq 0$ if and only if $\omegaSplitOne_i$ is carried isomorphically onto $\omegaSplitOne_{i+1}$, and similarly with $q_i^*$.
   \end{proof}
   If $\characteristic(k) \neq p$, any finite flat group scheme of order $p$ over $\Spec(k)$ is isomorphic to the constant group scheme $\Z/p\Z$. If $\characteristic(k) = p$, there are three finite flat group schemes $G$ of order $p$ over $\Spec(k)$ up to isomorphism: $\Z/p\Z$, $\mu_p$, and $\alpha_p$ \cite[Lemma 1]{OT}. The Cartier dual of $\Z/p\Z$ is $\mu_p$, and the Cartier dual of $\alpha_p$ is $\alpha_p$ itself. We record the dimensions of $\omega_G$ and $\omega_{G^*}$ in the table below.
   	\begin{center}
    \begin{tabular}{c|c|c|c}
	    $G$ & $\mu_p$ & $\mathbb{Z}/p\mathbb{Z}$ & $\alpha_p$\\
	    \hline
	    $(\dim_k \omega_{G}, \dim_k \omega_{G^*})$ & (1,0) & (0,1) & (1,1)
    \end{tabular}
    \captionof{table}{Dimensions of invariant differentials}
    \end{center}
   Thus knowing $q_i$ and $q_i^*$, one can determine the isomorphism type of $G_i$.
   \begin{corollary}
    \label{Oort-Tate generator locus}
    Let $y : S \rightarrow \MlocGL$. The support of the divisor on $S$ defined by the vanishing of $q_i^*$ is precisely the locus where the corresponding group scheme $G_i$ is infinitesimal.
   \end{corollary}
  \subsection{The integral model $\sAGL_1$}
   \label{sect the integral model sAGL_1}
   To define $\sAGL_1$, we will use Oort-Tate theory. We recall the concise summary \cite[Theorem 3.3.1]{HR}.
   \begin{theorem}
    \label{OT summary}\cite{OT}
    Let $OT$ be the $\Z_p$-stack representing finite flat group schemes of order $p$.
    \begin{compactenum}[(i)]
     \item OT is an Artin stack isomorphic to
      $$[(\Spec \Z_p[X,Y]/(XY - w_p))/\G_m]$$
     where $\G_m$ acts via $\lambda \cdot (X,Y) = (\lambda^{p-1}X, \lambda^{1-p}Y)$.  Here $w_p$ denotes an explicit element of $p\Z_p^\times$ given in \cite{OT}.
     \item The universal group scheme $\sG_{OT}$ over $OT$ is
      $$\sG_{OT} = [(\Spec_{OT}\sO[Z]/(Z^p - XZ))/\G_m],$$
    where $\G_m$ acts via $Z \rightarrow \lambda Z$, with identity section $Z=0$.
     \item Cartier duality acts on $OT$ by interchanging $X$ and $Y$.
    \end{compactenum}
   \end{theorem}

   As in \cite{HR}, we denote by $\sG_{OT}^\times$ the ``subscheme of generators'' defined by the ideal $(Z^{p-1} - X)$ in $\sG_{OT}$. The morphism $\sG_{OT}^\times \rightarrow OT$ is relatively representable, flat, and finite of degree $p-1$.

   \begin{definition} \cite[Section 3.3]{HR} (cf.\ \cite[5.1]{Pappas95} and \cite[1.8]{KatzMazur})
    \label{Oort-Tate generator}
    Let $\varphi : S \rightarrow OT$ be a morphism, $G = S \times_{OT} \sG_{OT}$, and $G^\times = S \times_{OT} \sG_{OT}^\times$. We say that a section $c \in G(S)$ is an Oort-Tate generator if $c \in G^\times(S)$.
   \end{definition}

   \begin{example}
    \label{Oort-Tate generator example}
    Let $\characteristic(k) = p$ and $G$ be a finite flat group scheme of order $p$ over $\Spec(k)$. Then the identity section $\varepsilon : \Spec(k) \rightarrow G$ is an Oort-Tate generator if $G \cong \mu_p$ or $G \cong \alpha_p$, but it is not an Oort-Tate generator if $G \cong \Z/p\Z$. The other $p-1$ sections of $\Z/p\Z$ are Oort-Tate generators.
   \end{example}
   
   The morphism $\varphi$ in the following definition is induced by the association of $\set{G_i}_{i=0}^{n-1}$ to a point of $\sAGL_0$ as described in the previous section.

   \begin{definition}
    \label{AGL_1 Definition}
    The scheme $\sAGL_1$ is the fibered product
    \begin{center}
     \begin{tikzcd}
      \sAGL_1 \arrow{r}\arrow{d}[swap]{\pi} & \sG_{OT}^\times \times_{\Z_p} \dotsb \times_{\Z_p} \sG_{OT}^\times \arrow{d}\\
      \sAGL_0 \arrow{r}{\varphi} & \overbrace{OT \times_{\Z_p} \dotsb \times_{\Z_p} OT}^{n \text{ times}}.
     \end{tikzcd}
    \end{center}
   \end{definition}
   The scheme $\sAGL_1$ admits the moduli description given in the introduction.
  \subsection{Local model of $\sAGL_1$}
   \label{sect local model of sAGL_1}
   Let $x : \Spec(k) \rightarrow \sAGL_0$ be a geometric point. By Theorem \ref{thm local model diagram}, there exists an \'{e}tale neighborhood $V \rightarrow \sAGL_0$ of $x$ and a section $\sigma : V \rightarrow \widetilde{\sAGL_0}$ of $\widetilde{\sAGL_0} \rightarrow \sA_0$ such that the composition $V \xrightarrow{\sigma} \widetilde{\sAGL_0} \xrightarrow{\Psi} \MlocGL$ is \'{e}tale. Set $\psi = \Psi \circ \sigma$. By restriction, we may assume that $\psi$ factors through an open subscheme $U \subset \MlocGL$ over which each $\sQ_i^*$ is trivial. Choosing a trivialization we identify each global section $q_i^*$ with a regular function on $U$. Since $\sQ_i^*$ is trivial, so is $\sQ_i$, and thus we also have $q_i \in \Gamma(U, \sO_U)$. Consider the diagram
    $$U \xleftarrow{\psi} V \rightarrow \sAGL_0 \xrightarrow{\varphi_i}  OT$$
   where $\varphi_i : \sAGL_0 \rightarrow OT$ is $\varphi$ followed by the $i$th projection. We present this $i$th factor as 
    $$OT = [(\Spec \Z_p[X_{i-1},Y_{i-1}]/(X_{i-1}Y_{i-1} - w_p))/\G_m].$$
   \begin{proposition}
    \label{Morphism to OT}
    Let $\rho_i : V \rightarrow OT$ be the morphism in the diagram above. It is given by
     $$\rho_i^*(X_i) = \varepsilon_i\psi^*(q_i^*) \quad \text{and} \quad \rho_i^*(Y_i) = w_p\varepsilon_i^{-1}\psi^*(q_i)$$
    where $\varepsilon_i$ is a unit in $V$ and $w_p$ is as in Theorem \ref{OT summary}.
   \end{proposition}
   \begin{proof}
    The special fiber of $\Mloc_\GL$, and therefore of $V$, is reduced \cite[Theorem 4.25]{GortzGL}. From the equalities $q_iq_i^* = p$ and $X_iY_i = \omega_pp$, the divisors defined by the vanishing of the global sections $Z(\psi^*(q_i^*))$ and $Z(\rho_i^*(X_i))$ are reduced. By Corollary \ref{Oort-Tate generator locus} and Example \ref{Oort-Tate generator example}, the locus where $\psi^*(q_i^*)$ vanishes agrees with the locus where $\rho_i^*(X_i)$ vanishes. Therefore $Z(\psi^*(q_i^*)) = Z(\rho_i^*(X_i))$.

    $\Mloc_\GL$ is of finite type over $\Spec(\Z_p)$, its generic fiber is smooth and hence normal, and it is flat over $\Spec(\Z_p)$ with reduced special fiber \cite[Theorem 4.25]{GortzGL}. It follows that $\Mloc_\GL$ is normal \cite[Proposition 8.2]{PZ}. Since $V \rightarrow \Mloc_\GL$ is \'{e}tale, we have that $V$ is normal as well. Thus the equality of divisors above implies $\psi^*(q_i^*)$ and $\rho_i^*(X_i)$ are equal up to a unit, say $\varepsilon_i$.  Similar statements apply to $\psi^*(q_i)$ and $\rho_i^*(Y_i)$, giving that $\psi^*(q_i)$ and $\rho_i^*(Y_i)$ are equal up to a unit. This unit must be $w_p\varepsilon_i^{-1}$ because $X_iY_i = w_pp$ and $q_iq_i^*=p$.
   \end{proof}

   \begin{proposition}
    \label{Local local model}
    Let $\widetilde{x} : \Spec(k) \rightarrow \sAGL_1$ be a geometric point and set $x= \pi(\widetilde{x})$.  Let $V \rightarrow \sAGL_0$ be an \'{e}tale neighborhood of $x$ which carries an \'{e}tale morphism $\psi : V \rightarrow \Mloc_\GL$ as described above. Suppose $V \rightarrow \Mloc_\GL$ factors through an affine open subscheme $U \subset \Mloc_\GL$ on which $\sQ_i^*$ is trivial for each $i$. Set 
     $$U_1 = \Spec_U\left(\sO[u_0, \dotsc, u_{n-1}]/\left(u_0^{p-1} - q_0^*, \dotsc, u_{n-1}^{p-1} - q_{n-1}^*\right)\right).$$
    Then there exists an \'{e}tale neighborhood $\widetilde{V}$ of $\widetilde{x}$ and an \'{e}tale morphism $\widetilde{\psi} : \widetilde{V} \rightarrow U_1$.
   \end{proposition}
   \begin{proof}
    Set $\widetilde{V} = V \times_{\sAGL_0} \sAGL_1$. Consider the diagram
    \begin{center}
     \begin{tikzcd}
      U_1 \arrow{d} & 
      \widetilde{V} \arrow{d} \arrow{r} \arrow[dashed]{l} &
      \sAGL_1 \arrow{d} \arrow{r}&
      \sG_{OT}^\times \times \dotsb \times \sG_{OT}^\times \arrow{d}\\
      U &
      V \arrow{l}\arrow{r} &
      \sAGL_0 \arrow{r} &
      OT \times \dotsb \times OT
     \end{tikzcd}
    \end{center}
    where two right squares are cartesian and denote by $\eta$ the morphism $\widetilde{V} \rightarrow \sG_{OT}^\times \times \dotsb \times \sG_{OT}^\times$. The morphism $\sG_{OT}^\times \rightarrow OT$ is relatively representable and thus $\widetilde{V}$ is isomorphic to
     $$\Spec_{V}\left(\sO[u_0, \dots, u_{n-1}]/(u_0^{p-1} - \rho_0^*(X_0), \dotsc, u_{n-1}^{p-1} - \rho_{n-1}^*(X_{n-1}))\right).$$
    With the notation as in the previous proposition, by shrinking $V$ we can choose for each $i$ a $(p-1)$th root of $\varepsilon_i$, denote it by $\tau_i$. Let $\widetilde{V} \rightarrow U_1$ be the morphism induced by the ring homomorphism $\Gamma(U_1, \sO_{U_1}) \rightarrow \Gamma(\widetilde{V}, \sO_{\widetilde{V}})$ which sends $u_i$ to $\tau_iu_i$. This is well-defined by the previous proposition and it follows that $\widetilde{V} \cong U_1 \times_U V$. Therefore the morphism $\widetilde{V} \rightarrow U_1$ is \'{e}tale.
   \end{proof}
   \begin{remark}
    Given a covering of affine open subschemes $\set{U_j}$ of $\MlocGL$ such that $\sQ_i^*$ is trivial on each $U_j$ for every $i$, it is tempting to hope that one may glue together the schemes defined in the proposition above to get a connected scheme $\Mloc_1$ which is an \'{e}tale local model of $\sAGL_1$. However this is not possible. Indeed, suppose that $\Mloc_1$ is such a scheme with respect to the a cover $\set{U_j}$ of $\MlocGL$. Then for each $j$, $U_j \times_{\MlocGL} \Mloc_1$ is isomorphic to
     $$\Spec_{U_j}\left(\sO_{U_j}[u_0, \dotsc, u_{n-1}]/(u_0^{p-1}-q_0^*, \dotsc, u_{n-1}^{p-1}-q_{n-1}^*)\right).$$
    The sections $q_0^*, \dotsc, q_{n-1}^*$ vanish only on the special fiber. Therefore the restriction of $\Mloc_1 \rightarrow \Mloc_\GL$ to the generic fiber is finite \'{e}tale. However the generic fiber of $\Mloc_\GL$ is the Grassmannian $\Gr(n,r)$ which is simply connected. This easily leads to a contradiction.
   \end{remark}
   
   To construct a local model of $\sAGL_1$, we will use an affine open subscheme of $\MlocGL$ that serves as a local model of $\sAGL_0$. Let $\sF_\SL$ denote the affine flag variety associated with $\SL_n$, noting that the group scheme $\sG$ from Section \ref{Shimura Datum and Moduli Description} acts on $\sF_\SL$ \cite[Section 4]{GortzGL}. The next theorem is extracted from \cite{GortzGL}.
   
   \begin{theorem}
    There is a $\sG$-equivariant embedding $\MlocGL \otimes \F_p \hookrightarrow \sF_\SL$. The stratification of $\sF_\SL$ by Schubert cells induces a stratification of $\MlocGL \otimes \F_p$. There is a unique stratum of $\MlocGL \otimes \F_p$ which consists of a single closed point, called the ``worst point''. Any open subscheme of $\MlocGL$ containing the worst point is an \'{e}tale local model of $\sA_0$.
   \end{theorem}
   
   A particular open subscheme is defined in \cite{GortzGL} as follows. An alcove of $\SL_n$ is given by a collection $\set{x_0, \dots, x_{n-1}}$ where $x_i \in \Z^n$ written as $(x_i(1), \dots, x_i(n))$ satisfying the following two conditions. Setting $x_n(j) = x_0(j)+1$, for $0 \leq i \leq n-1$ we require $x_i(j) \leq x_{i+1}(j)$ for all $1 \leq j \leq n$ and $\sum_{j=1}^n x_{i+1}(j) = 1 + \sum_{j=1}^n x_i(j)$. Consider the alcoves
    $$\omega = (\omega_0, \dots, \omega_{n-1}) \quad \text{with} \quad \omega_i = (1^i, 0^{n-i})$$
    $$\tau = ((1^r, 0^{n-r}), (1^{r+1}, 0^{n-r-1}), \dotsc, (2^{r-1}, 1^{n-r+1})).$$
   As defined in \cite{KR}, the size of an alcove $x$ is $\sum_j x_0(j) - \omega_0(j)$ and that $x$ is $\mu$-admissible means $x \leq \sigma(\omega)$ for some $\sigma \in S_n$, where $\leq$ is the Bruhat order.
   
   The affine Weyl group of $\SL_n$ is isomorphic to $S_n$, the symmetric on $n$ letters, and acts simply transitively on the collection of alcoves of size $r$. Thus fixing the base alcove $\tau$ we may identify the two sets. Let $\set{e_1, \dots, e_n}$ be the canonical basis of $\Z_p^n$ and fix the basis $\set{e_j^i}$ of the $\Z_p$-module $\LambdaSplitOne_i$ where
    $$e_j^i = p^{-1}e_j \quad \text{for} \quad 1 \leq j \leq i \qquad \text{and} \qquad e_j^i = e_j \quad \text{for} \quad i < j \leq n.$$

   \begin{definition}\cite[Definition 4.4]{GortzGL}
    For a $\mu$-admissible alcove $x = (x_0, \dotsc, x_{n-1})$ we define an open subscheme of $\MlocGL$, denoted $U_x^\GL$, consisting of the points $(\omegaSplitOne_i)_i$ such that for all $i$ the quotient $\LambdaSplitOne_i/\omegaSplitOne_i$ is generated by those $e_j^i$ with $\omega_i(j) = x_i(j)$.
   \end{definition}
   The open subscheme $U_\tau^\GL$ contains the ``worst point'' and hence serves as an \'{e}tale local model of $\sAGL_0$. Henceforth we shall denote $U_\tau^\GL$ by $U_0^\GL$. It is immediate that the line bundles $\sQ_i$ and $\sQ_i^*$ are trivial over $U_0^\GL$.

   \begin{theorem}
    \label{Main theorem GL_n}
    Choose a trivialization of each $\sQ_i|_{U_0^\GL}$ and identify $q_i$ and $q_i^*$ with regular functions on $U_0^\GL$. The scheme 
     $$U_1^\GL = \Spec_{U_0^\GL}\left(\sO[u_0, \dotsc, u_{n-1}]/\left(u_0^{p-1} - q_0^*, \dotsc, u_{n-1}^{p-1} - q_{n-1}^*\right)\right)$$
    is an \'{e}tale local model of $\sAGL_1$.
   \end{theorem}

   To make the above theorem completely explicit, we now describe the presentation of $U_0^\GL$ computed in \cite{GortzGL}. As $\MlocGL$ is a closed subscheme of the $n$-fold product of $\Gr(n,r)$, we represent a point of $\MlocGL$ by $\set{\omegaSplitOne_i}_{i=0}^{n-1}$ where each subspace $\omegaSplitOne_i$ is $r$-dimensional, and we represent $\omegaSplitOne_i$ as the column space of the $n \times r$ matrix $(a_{jk}^i)$. That $(\omegaSplitOne_i)_i \in U_0^\GL$ implies the $r \times r$ submatrix given by rows $i+1$ to $r+i$ (taken cyclically, so row $n+1$ is row 1) of $\omegaSplitOne_i$ is invertible for $0 \leq i < n$.  We normalize by requiring this submatrix to be the identity matrix.  For example, $\omegaSplitOne_0$ and $\omegaSplitOne_1$ are represented by
    $$\begin{pmatrix}
       1 & \\
         & \ddots \\
         & & 1\\
       a_{11}^0 & \dotso & a_{1r}^0\\
       \vdots & & \vdots\\
       a_{n-r,1}^0 & \dotso & a_{n-r,r}^0
      \end{pmatrix}
      \quad \text{and} \quad
      \begin{pmatrix}
       a_{n-r,1}^1 & \dotso & a_{n-r,r}^1\\
       1 & \\
       & \ddots \\
       & & 1\\
       a_{11}^1 & \dotso & a_{1r}^1\\
       \vdots &  & \vdots\\
       a_{n-r-1,1}^1 & \dotso & a_{n-r-1,r}^1
      \end{pmatrix}.
    $$
   From the normalization, the entries of the matrix are uniquely determined by its column space. By abuse of notation, we will use $\omegaSplitOne_i$ to denote both the subspace and the matrix representing it. The condition $\omegaSplitOne_i$ is mapped into $\omegaSplitOne_{i+1}$ is expressed as $\varphi_i(\omegaSplitOne_i) = \omegaSplitOne_{i+1} A_i$ for some $r \times r$ matrix $A_i$. From the normalization, $A_i$ is determined:
    $$A_i = 
      \begin{pmatrix}
       0 & 1 \\
	 & \ddots & \ddots \\
	 & & 0 & 1\\
       a_{11}^i & a_{12}^i & \dotso & a_{1r}^i
      \end{pmatrix}.
    $$
   As shown in \cite[Proposition 4.13]{GortzGL}, $U_0^\GL \cong \Spec(B_{\GL})$ with 
    $$B_{\GL} = \Z_p[a_{1k}^i; i = 0, \dotsc, n-1, k=1, \dotsc, r]/I$$
   where $I$ is the ideal generated by the entries of the matrices
    $$A_{n-1}\dotsm A_0 - p \cdot I, \ A_{n-2}\dotsm A_0 A_{n-1} - p \cdot I, \ \dotsc, \ A_0 A_{n-1} \dotsm A_1 - p \cdot I.$$    
   \begin{lemma}\label{lemma q_i and q_i^*}
    Let $\set{\sF_i \subset \LambdaSplitOne_{i,S}}$ define a geometric point $\Spec(k) \rightarrow U_0^\GL$. With the presentation of $U_0^\GL$ described above,
    \begin{compactenum}[(i)]
     \item the map $\sF_i \rightarrow \sF_{i+1}$ is an isomorphism if and only if $a_{11}^i \neq 0$; and
     \item the map $\LambdaSplitOne_i / \sF_i \rightarrow \LambdaSplitOne_{i+1} / \sF_{i+1}$ is an isomorphism if and only if $a_{n-r,r}^{i+1} \neq 0$, where the upper index is taken modulo $n$ with the standard representatives $\set{0, \dots, n-1}$.
    \end{compactenum}
   \end{lemma}
   \begin{proof}
    For (i), the map $\omegaSplitOne_i \rightarrow \omegaSplitOne_{i+1}$ is an isomorphism if and only if $\det(A_i) = \pm a_{11}^i \neq 0$. To show (ii), we may take $\set{\overline{e}_{i+r+1}^i, \dotsc, \overline{e}_{i+n}^i}$ to be a basis of $\LambdaSplitOne_i/\omegaSplitOne_i$, where $\overline{e}_{j}^i$ denotes the reduction of $e_{j}^i$ modulo $\omegaSplitOne_i$. The upper index is taken modulo $n$ with the standard set of representatives $\set{0, \dots, n-1}$ while the lower index is taken modulo $n$ with the representatives $\set{1, \dots, n}$. Then the matrix representing $\LambdaSplitOne_i/\omegaSplitOne_i \rightarrow \LambdaSplitOne_{i+1}/\omegaSplitOne_{i+1}$ with respect to these bases is
     $$\begin{pmatrix}
	-a_{1r}^{i+1} & 1\\
	\vdots & & \ddots\\
	-a_{n-r-1,r}^{i+1}&&&1\\
	-a_{n-r,r}^{i+1} & 0 & \dots & 0
       \end{pmatrix}.$$

    Hence the map $\LambdaSplitOne_i/\omegaSplitOne_i \rightarrow \LambdaSplitOne_{i+1}/\omegaSplitOne_{i+1}$ is an isomorphism if and only if $a_{n-r,r}^{i+1} \neq 0$.
   \end{proof}

   It follows from the lemma that, up to a unit, $q_i^* = a_{n-r,r}^{i+1}$ for $0 \leq i \leq n-1$. Combining this with Theorem \ref{Main theorem GL_n} we get the following.
   \begin{theorem}
    Denote by $U_1^\GL$ the spectrum of
     $$B_\GL[u_0, \dotsc, u_{n-1}]/(u_0^{p-1} - a_{n-r,r}^1, \dotsc, u_{n-2}^{p-1} - a_{n-r,r}^{n-1}, u_{n-1}^{p-1} - a_{n-r,r}^0).$$
    Then $U_1^\GL$ is an \'{e}tale local model of $\sAGL_1$.
   \end{theorem}
 \subsection{The general case}
  \label{The General Case}
  We have thus far assumed that $F^+ = \Q$ and $D$ splits over $\Q_p$. In this section, we extend the results to the more general case where $F^+$ is a totally real finite extension of $\Q$.

  Let $F$ be an imaginary quadratic extension of $F^+$, with $p$ an odd rational prime totally split in $F^+$ such that each factor of $(p)$ in $F^+$ splits in $F$. Write $(p) = \prod_jp_j$ in $F^+$ and $p_j = \p_j\p_j^*$ in $F$. Then $F_{\Q_p} = \prod_j F_{\p_j} \times F_{\p_j^*}$ making $D_{\Q_p} = \prod_j D_{\p_j} \times D_{\p_j^*}$. Here $D_{\p_j}$ and $D_{\p_j^*}$ are each respectively a central simple $F_{\p_j}$- and $F_{\p_j^*}$-algebra with the involution $*$ giving $D_{\p_j} \cong D_{\p_j^*}^\opp$. We furthermore assume that $D_{\Q_p}$ splits after an unramified extension $\widetilde{F}$ of $\Q_p$. The splitting of $D_{\Q_p}$ gives $G_{\widetilde{F}} = \prod_j G_j$ where the factor $G_j(R)$ may be written as
   $$\set{(x_1,x_2) \in (D_{\p_j, \widetilde{F}} \otimes R)^\times \times (D_{\p_j^*, \widetilde{F}} \otimes R)^\times : x_1 = c(x_2^*)^{-1} \text{ for some } c \in R^\times}$$
  for an $\widetilde{F}$-algebra $R$. Then $G_j \cong D_{\p_j,\widetilde{F}}^\times \times \G_m \cong \GL_n \times \G_m$. Let $\mu_j$ be the cocharacter $\mu : \G_{m,\overline{\Q}_p} \rightarrow G_{\overline{\Q}_p}$ composed with the $j$th projection in the decomposition above. Then $\mu = \prod_j \mu_j$. The periodic lattice chain $\sL$ is taken to be the product over $j$ of those describe in Section \ref{Shimura Datum and Moduli Description}.
  
  With this data the integral model of the Shimura variety $\sAGL_0$ is the scheme representing the moduli problem given in \cite[Definition 6.9]{RZ}, similar to the moduli problem in Section \ref{Shimura Datum and Moduli Description}. We then define $\sAGL_1$ as in Definition \ref{AGL_1 Definition}. After an unramified base extension, the local model $\Mloc$ as given in Definition \ref{def RZ local model} is a product of the local models in the case $F^+ = \Q$, indexed by the factors of $(p)$ in $F^+$. Let $\Mloc_j$ denote the $j$th factor. For $0 \leq i \leq n-1$, we have the universal line bundle $\sQ_{j,i}^*$ on $\Mloc_j$ with the global section $q_{j,i}^*$ as defined in Section \ref{sect the group schemes G_i}. Let $U = \prod_j U_j$ be an affine open subscheme of $\MlocGL$ such that $U_j \subset \Mloc_j$ and each $\sQ_{j,i}^*$ is trivial on $U_j$. Choosing a trivialization, we identify the $q_{j,i}^*$ with regular functions on $U_j$.
 
 \begin{theorem}
  Let $U_1 = \prod_j U_{1,j}$, where
   $$U_{1,j} = \Spec_{U_j}\bigg(\sO[u_0, \dots, u_{n-1}]/(u_0^{p-1} - q_{j,0}^*, \dots, u_{n-1}^{p-1} - q_{j,n-1}^*)\bigg).$$
  Then $U_1$ is an \'{e}tale local model of $\sAGL_1$.
 \end{theorem}

 \section{Integral and local model of $\sAGSp_1$ and $\sAGSp_{1,\bal}$}
  \label{Local model of sAGSp1}
  Let $\set{e_1, \dots, e_{2n}}$ be the standard basis of $V = \Q_p^{2n}$ and equip $V$ with the standard symplectic pairing $(\cdot, \cdot)$ where $(e_i, e_{2n+1-j}) = \delta_{ij}$. Define the $\Z_p$-lattice chain $\Lambda_0 \subset \Lambda_1 \subset \dots \subset \Lambda_{2n-1}$ where
   $$\Lambda_i = \generate{p^{-1}e_1, \dots, p^{-1}e_i, e_{i+1}, \dots, e_{2n}} \subset \Q_p^{2n} \text{ as a } \Z_p\text{-module}$$
  and extend it periodically to all $i \in \Z$ by $\Lambda_{i+2n} = p^{-1}\Lambda_i$. Then $\sL = (\Lambda_i)_i$ is a full periodic self-dual lattice chain with $\Lambda_i^\perp = \Lambda_{-i}$. Take $\G = \GSp_{2n}$ with minuscule cocharacter $\mu = (0^n, (-1)^n)$. The reflex field at $p$ is $\Q_p$, $K^p$ is a sufficiently small compact open subgroup of $\G(\A_f^p)$, and $K_p = \Aut(\sL)$, the automorphism group of the polarized multichain $\sL$. 
  
  With this data, the moduli problem \cite[Definition 6.9]{RZ} for the integral model $\sAGSp_0$ is equivalent to the following. For any $\Z_p$-scheme $S$, $\sAGSp_0(S) = \sAGSp_{0,K^p}(S)$ is the set of tuples $(A_\bullet, \lambda_0, \lambda_n, \bar{\eta})$, up to isomorphism, where
  \begin{compactenum}[(i)]
   \item $A_\bullet$ is a chain $A_0 \xrightarrow{\alpha_0} A_1 \xrightarrow{\alpha_1} \dotsa \xrightarrow{\alpha_{n-1}} A_n$ of $n$-dimensional abelian schemes over $S$ with each morphism $\alpha_i$ an isogeny of degree $p$;
   \item the maps $\lambda_0 : A_0 \rightarrow \hat{A}_0$ and $\lambda_n : A_n \rightarrow \hat{A}_n$ are principal polarizations making the loop starting at any $A_i$ or $\hat{A}_i$ in the diagram
   \begin{center}
    \begin{tikzcd}
     A_0 \arrow{r}{\alpha_0} &
     A_1 \arrow{r}{\alpha_1} &
     \dotsa \arrow{r}{\alpha_{n-1}} &
     A_n \arrow{d}{\lambda_n} \\
     \hat{A}_0 \arrow{u}{\lambda_0^{-1}}&
     \hat{A}_1 \arrow{l}[swap]{\alpha_0^\vee} &
     \dotsa \arrow{l}[swap]{\alpha_1^\vee} &
     \hat{A}_n \arrow{l}[swap]{\alpha_{n-1}^\vee}
    \end{tikzcd}
   \end{center}
   multiplication by $p$; and
   \item $\bar{\eta}$ is a $K^p$-level structure on $A_0$ (see \cite[Section 5]{Ko92} for details).
  \end{compactenum}

  There is an alternative description of this moduli problem in terms of chains of finite flat group subschemes instead of chains of isogenies \cite[Section 1]{dJ}. The local model $\Mloc_\GSp$ is the $\Z_p$-scheme representing the following functor.  An $S$-valued point of $\Mloc_\GSp$ is given by a commutative diagram
  \begin{center}
   \begin{tikzcd}
    \Lambda_{0,S} \arrow{r} &
    \Lambda_{1,S} \arrow{r} &
    \dotsa \arrow{r} &
    \Lambda_{2n-1,S} \arrow{r} &
    p^{-1}\Lambda_{0,S} \\
    \omegaSplitOne_0 \arrow{r} \arrow[hook]{u}&
    \omegaSplitOne_1 \arrow{r} \arrow[hook]{u}&
    \dotsa \arrow{r} &
    \omegaSplitOne_{2n-1} \arrow{r} \arrow[hook]{u}&
    p^{-1}\omegaSplitOne_0 \arrow[hook]{u}
   \end{tikzcd}
  \end{center}
  where $\Lambda_{i,S}=\Lambda_i \otimes_{\Z_p} \sO_S$, the morphisms $\Lambda_{i,S} \rightarrow \Lambda_{i+1,S}$ are induced by the inclusions $\Lambda_i \subset \Lambda_{i+1}$, $\omegaSplitOne_i$ are locally free $\sO_S$-submodules of rank $n$ which are Zariski-locally direct summands of $\Lambda_{i,S}$, and the $\omegaSplitOne_i$ satisfy the following duality condition: the map $\omegaSplitOne_i \rightarrow \Lambda_{i,S} \xrightarrow{\sim} \hat{\Lambda}_{2n-i,S} \rightarrow \hat{\omegaSplitOne}_{2n-i}$ is zero for all $i$. Here $\hat{\cdot} = \sheafhom_{\sO_S}(\cdot, \sO_S)$ and $\Lambda_{i,S} \xrightarrow{\sim} \hat{\Lambda}_{2n-i,S}$ is induced from the duality $\Lambda_i \xrightarrow{\sim} \hat{\Lambda}_{2n-i}$. It follows that for $0 \leq i \leq n-1$, $\sF_{2n-i}$ is determined by $\sF_i$.

  $\Mloc_\GSp$ is readily seen to be representable as a closed subscheme of $\Mloc_\GL$. Using the description of the open subscheme $U_0^\GL = \Spec(B_\GL) \subset \Mloc_\GL$ from Section \ref{sect local model of sAGL_1}, the duality condition imposes the following additional equations for $1 \leq i \leq n-1$ \cite[5.1]{GortzGSp}.
   $$a_{jk}^{2n-i} = \varepsilon_{jk}a_{n-k+1,n-j+1}^{i} \quad  \text{with} \quad \varepsilon_{jk} = \left\{ \begin{array}{cl} 1 & j,k \leq i \text{ or } j,k \geq i+1\\ -1 & \text{otherwise}\end{array} \right.$$
  We denote the resulting ring by $B_\GSp$ and set $U_0^\GSp = \Spec(B_\GSp)$. For an $S$-valued point of $\sA_0^\GSp$, let $\set{G_i}_{i=0}^{n-1}$ and $\set{G_i^*}_{i=0}^{n-1}$ be the collection of group schemes over $S$ where $G_i = \ker(A_i \rightarrow A_{i+1})$ and $G_i^* = \ker(\hat{A}_{i+1} \rightarrow \hat{A}_i)$. Note that $G_i$ is the Cartier dual of $G_i^*$. Then the results from Section \ref{sect the group schemes G_i} on the dimension of the invariant differentials of the group schemes carry over. The association of an $S$-valued point with the collection $\set{G_0, \dots, G_{n-1}, G_{n-1}^*, \dots, G_0^*}$ of group schemes induces the morphism
   $$\sAGSp_0 \rightarrow \overbrace{OT \times \dotsb \times OT}^{2n \text{ times}}.$$
  
  \subsection{$\Iw_1(p)$-level subgroup}
   \label{subsect Iw_1(p)-level subgroup}
			We say that $K = K_pK^p$ is an $\Iw_1(p)$-level subgroup if $K_p$ stabilizes the lattice chain $\set{\Lambda_i}$ and pointwise fixes $\Lambda_{i+1}/\Lambda_i$ for $0 \leq i \leq n-1$. The resulting Shimura variety admits the moduli description stated in the introduction and we therefore define the integral model as follows.			 
			\begin{definition}
			$\sAGSp_1$ is the fibered product
			\begin{center}
					\begin{tikzcd}
							\sAGSp_1 \arrow{r}\arrow{d}[swap]{\pi} & \sG_{OT}^\times \times \dotsb \times \sG_{OT}^\times \arrow{d}\\
							\sAGSp_0 \arrow{r}{\varphi} & \overbrace{OT \times \dotsb \times OT}^{n \text{ times}}.
					\end{tikzcd}
				\end{center}
				where the bottom arrow is $\sAGSp_0 \rightarrow \overbrace{OT \times \dotsb \times OT}^{2n \text{ times}}$ followed by the projection onto the first $n$ factors.
			\end{definition}

			We define $\sQ_i$, $\sQ_i^*$ and their global sections $q_i, q_i^*$ as in Section \ref{sect the group schemes G_i}. As $\MlocGSp$ is flat over $\Spec(\Z_p)$ and the special fiber is reduced \cite{GortzGSp}, the results from Sections 3 and 4 carry over as well in the appropriate manner.

			\begin{theorem}
			\label{main theorem for GSp with Iw_1(p) level}
			Denote by $U_1^\GSp$ the spectrum of
					$$B_\GSp[u_0, \dotsc, u_{n-1}]/(u_0^{p-1} - a_{n-r,r}^1, \dotsc, u_{n-2}^{p-1} - a_{n-r,r}^{n-1}, u_{n-1}^{p-1} - a_{n-r,r}^0).$$
			Then $U_1^\GSp$ is an \'{e}tale local model of $\sAGSp_1$.
			\end{theorem}
  \subsection{$\Iw_1^\bal(p)$-level subgroup}
   \label{subsect Iw_1^bal(p)-level subgroup}
   We say that $K = K_pK^p$ is an $\Iw_1^\bal(p)$-level subgroup if $K_p$ is the pro-unipotent radical of an Iwahori. The resulting Shimura variety admits the following moduli description. With $S$ a scheme over $\Q_p$, to give an $S$-valued point of $\sAGSp_{1,\bal} \otimes \Q_p$ it is equivalent to an $S$-valued point of $\sAGSp_0 \otimes \Q_p$ and the collections $\set{g_i}_{i=0}^{n-1}$ and $\set{g_i^*}_{i=0}^{n-1}$. Here, $g_i$ and $g_i^*$ are Oort-Tate generators of the group schemes $G_i$ and $G_i^*$ respectively and are subject the the condition that that $g_ig_i^*$ is independent of $i$. Note that $(g_ig_i^*)^{p-1} = \omega_p$, and thus $\set{g_0, \dots, g_{n-1}}$ determines $\set{g_0, \dots, g_{2n-1}}$ once a $(p-1)$th root of $\omega_p$ is chosen. Therefore letting $\sAGSp_0[t] = \Spec_{\sAGSp_0}(\sO[t]/(t^{p-1} - \omega_p))$, we define the integral model as follows.
   
			\begin{definition}
			$\sAGSp_{1,\bal}$ is the fibered product
			\begin{center}
					\begin{tikzcd}
							\sAGSp_1 \arrow{r}\arrow{d}[swap]{\pi} & \sG_{OT}^\times \times \dotsb \times \sG_{OT}^\times \arrow{d}\\
							\sAGSp_0[t] \arrow{r}{\varphi} & \overbrace{OT \times \dotsb \times OT}^{n \text{ times}}
					\end{tikzcd}
					\end{center}
					where the bottom morphism factors through the projection $\sAGSp_0[t] \rightarrow \sAGSp_0$.
			\end{definition}

			We define $\sQ_i$, $\sQ_i^*$ and their global sections $q_i, q_i^*$ as in Section \ref{sect the group schemes G_i}.

			\begin{theorem}
				\label{main theorem for GSp with Iw_1^bal(p) level}
				The scheme $U_{1,\bal}^{\GSp} = \Spec_{U_1^\GSp}(\sO[t]/(t^{p-1} - p))$	is an \'{e}tale local model of $\sAGSp_{1,\bal}$.
			\end{theorem}
		Note that for all $n \geq 1$, $U_{1,\bal}^\GSp$ is not normal. In particular, $\frac{t}{u_i}$ is not a regular section of $\sO_{U_{1,\bal}^\GSp}$ for any $i$ but
			$$\left(\frac{t}{u_i}\right)^{p-1} = \frac{p}{a_{n-r,r}^{i+1}} = a_{11}^{i}$$
		since $a_{11}^ia_{n-r,r}^{i+1} = p$. Thus in constructing the normalization of $U_{1,\bal}^{\GSp}$, we must at least adjoin elements $u_i^*$ for $0 \leq i \leq n-1$ satisfying $u_iu_i^* = t$ and $(u_i^*)^{p-1} = a_{11}^i$ to the coordinate ring $\sO_{U_{1,\bal}^\GSp}$. By adjoining the elements and relations, the resulting scheme has been shown to be normal \cite{HS} and hence gives the normalization. This will be used in Section \ref{sect Iw_1^bal(p) resolution}.

 \section{A Resolution of $\sAGSp_1$ for $n=2$}
  \label{A Resolution of sAGSp_1 for GSp_4}
  In this section we will only be utilizing the integral and local models associated with $\G = \GSp_4$ and $K^p = K(N)$, where the integer $N$ is such that $p \nmid N$ and $N \geq 3$. As such, all unnecessary subscripts and superscripts will be removed from the notation. Let $\breve E$ denote the completion of the maximal unramified extension of the reflex field $E = \Q_p$. All schemes in this section are of finite type over $\Spec(\sO_{\breve E})$.
  
  The connected components of $\sA_0 \otimes_{\Z_p} \sO_{\breve E}$ are indexed by the primitive $N$th roots of unity \cite[Section 2.1]{HainesConnectedComponents}. The number of connected components of the KR-strata (see below) in each connected component of $\sA_0 \otimes_{\Z_p} \sO_{\breve E}$ is the same. As will be seen, it thus suffices to consider a single connected component of $\sA_0 \otimes_{\Z_p} \sO_{\breve E}$. We abuse the notation by denoting such a connected component of $\sA_0 \otimes_{\Z_p} \sO_{\breve E}$ by $\sA_0$. Likewise we write $\sA_1$ for the inverse image of this connected component with respect to the morphism $\sA_1 \otimes_{\Z_p} \sO_{\breve E} \rightarrow \sA_0 \otimes_{\Z_p} \sO_{\breve E}$ and $\Mloc$ for $\Mloc \otimes_{\Z_p} \sO_{\breve E}$. Note that $\breve E \otimes_{\Z_p} \F_p$ is an algebraic closure of $\F_p$, written as $\overline{\F}_p$, and that $\sA_0 \otimes \overline{\F}_p$ is connected \cite[Lemma 13.2]{H}.

  At each stage of the construction of the resolution, we will first work on a local model and then carry this over to the corresponding integral model by producing a ``linear modification'' in the sense of \cite[Section 1]{P}. Thus when performing a blowup of a local model, we require that the subscheme being blown up correspond to a subscheme of the integral model. In order to understand more of the global structure of the resolution, such as the number of irreducible components and how they intersect, it is also necessary to track how certain subschemes transform with each modification. 
  
  For a subscheme $Z \subset X$, $Z^\red$ denotes the corresponding reduced scheme. For a morphism of schemes $f : Y \rightarrow X$, $f^{-1}(Z)$ means the scheme-theoretic pullback $Y \times_X Z$ unless otherwise noted. 
  
  \begin{definition}
   \cite[Definition IV-15]{EisenbudHarris}
   Let $X$ be a scheme, $\sZ \subset X$ a subscheme. We say that $\sZ$ is Cartier at a closed point $p$ of $X$ if in an affine open neighborhood of $p$, $\sZ$ is the zero locus of a single regular function which is not a zero divisor. We say that $\sZ$ is a Cartier subscheme of $X$ if $\sZ$ is Cartier at all closed points of $X$.
  \end{definition}
  
  \begin{definition}
  Let $\rho : X' \rightarrow X$ be a modification, i.e.\ a proper birational morphism.
  \begin{compactenum}
   \item If $\rho$ is given by the blowup of a closed subscheme $\sZ$ of $X$, then $\sZ$ is called a center of $\rho$.  
   \item The true center of $\rho$ is the closed reduced subscheme of $X$ given set-theoretically by the complement of the maximal open subscheme where $\rho$ is an isomorphism.
   \item The fundamental center of $\rho$ is the reduced subscheme of $C_\rho$ whose support is given by the closed points with fiber of dimension at least one.
   \item The residual locus of $\rho$ is the complement $C_\rho \setminus C_\rho^\fund$.
   \item The exceptional locus of $\rho$ is $\rho^{-1}(C_\rho)^\red$.
   \item The strict transform of a subscheme $\sW \subset X$ is either $\rho^{-1}(\sW)^\red$ if $\sW \subset C_\rho$ or the Zariski closure of $\rho^{-1}(\sW \setminus C_\rho)$ inside of $X'$ with reduced scheme structure if $\sW \not\subset C_\rho$.
  \end{compactenum}
 \end{definition}
 \begin{remark}
  When $\rho$ is a modification with a center $\sZ$, we will often say ``the center'' when there is a canonical choice of $\sZ$. By upper semi-continuity of the dimension of the fiber, the fundamental center is a closed subscheme of the true center. Since all schemes are of finite type over $\Spec(\sO_{\breve E})$, the fiber over a closed point of the residual locus is a finite collection of closed points.
 \end{remark}
  
  \begin{definition}
   Let $M$ be an \'{e}tale local model of $X$. Fix a triple $(V, \varphi, \psi)$ where $\varphi : V \rightarrow X$ is an \'{e}tale cover and $\psi : V \rightarrow M$ is an \'{e}tale morphism. We say that a subscheme $Z \subset M$ \'{e}tale locally corresponds to a subscheme $\sZ \subset X$ with respect to $(V, \varphi, \psi)$ if $\varphi^{-1}(\sZ) = \psi^{-1}(Z)$ as subschemes of $V$.
  \end{definition}
  
  We gather here some facts that will be used implicitly throughout the construction. With $(V, \varphi, \psi)$ fixed as in the definition above, if $\sZ_1, \sZ_2 \subset X$ \'{e}tale locally correspond to $Z_1, Z_2 \subset M$ respectively, then so do their union, intersection, complement, and Zariski-closure. Also $\Bl_{Z_1}(M)$ is an \'{e}tale local model of $\Bl_{\sZ_1}(X)$ by taking 
   $$V' := \Bl_{\sZ_1}(X) \times_X V = \Bl_{Z_1}(M) \times_M V$$
  with $\varphi' : V' \rightarrow \Bl_{\sZ_1}(X)$ and $\psi' : V' \rightarrow \Bl_{Z_1}(M)$ the pullbacks of $\varphi$ and $\psi$ respectively. Under an \'{e}tale morphism, the pullback of the locus where a subscheme is Cartier is precisely the locus where the pullback of the subscheme is Cartier. Thus with respect to $(V, \varphi, \psi)$ and the induced $(V', \varphi', \psi')$, the true centers, fundamental centers, residual loci, and exceptional loci \'{e}tale locally correspond with respect to the morphisms $\rho_X : \Bl_{\sZ}(X) \rightarrow X$ and $\rho_M : \Bl_Z(M) \rightarrow M$. Furthermore, it is easy to show that $\ST_{\rho_X}(\sZ_2)$ \'{e}tale locally corresponds to $\ST_{\rho_M}(Z_2)$, again with respect to $(V', \varphi', \psi')$.
  
  The next lemma roughly says that in each step of our construction, the fiber of a morphism can be seen \'{e}tale locally. That it is applicable to every step will become apparent.
  \begin{lemma}
   \label{fiber correspondence}
   Let $M$ and $M'$ be \'{e}tale local models of $X$ and $X'$ respectively. Suppose there is an \'{e}tale cover $\varphi : V \rightarrow X$ with an \'{e}tale morphism $\psi : V \rightarrow M$ along with the morphisms $\rho_X$, $\rho_V$, and $\rho_M$ giving the diagram
   \begin{center}
    \begin{tikzcd}
     X' \arrow{d}[swap]{\rho_X}&
     V' \arrow{l}[swap]{\varphi'} \arrow{r}{\psi'} \arrow{d}{\rho_V} &
     M' \arrow{d}{\rho_M}\\
     X & 
     V \arrow{l}[swap]{\varphi} \arrow{r}{\psi} &
     M
    \end{tikzcd}
   \end{center}
   where the left and right squares are cartesian. Let $v$ be a closed point of $V$ and set $x = \varphi(v)$ and $y = \psi(v)$. Then $\rho_X^{-1}(x) \cong \rho_M^{-1}(y)$.
  \end{lemma}
  \begin{proof}
   Since $\varphi$ and $\psi$ are \'{e}tale, $k(x) = k(v) = k(y)$ where $k(\cdot)$ denotes the residue field. Thus 
   \begin{align*}
    \rho_X^{-1}(x) &= X' \times_X \Spec(k(x)) = (X' \times_X V) \times_V \Spec(k(x))\\
    &= (M' \times_M V) \times_V \Spec(k(x)) = M' \times_M \Spec(k(x))\\
    &= (M' \times_M \Spec(k(y)) ) \times_{\Spec(k(y))} \Spec(k(x))\\
    &= \rho_M^{-1}(y) \times_{\Spec(k(y))} \Spec(k(x))\\
    &= \rho_M^{-1}(y).
   \end{align*}
  \end{proof}

  Roman letters such as $Z$, $C$, and $E$ will be used to denote subschemes of the local models. Calligraphic letters such as $\sZ$, $\sC$, and $\sE$ denote subschemes of the integral models that \'{e}tale locally correspond to their Roman counterparts. We will use $Z_{ij}$ for irreducible components, $Z_i$ for the center of blowups, $C_i$ for the true centers, and $E$ for the exceptional loci. The schemes constructed in each additional step will be decorated with an additional tick mark $'$, and the superscript $[i]$ denotes $i$ tick marks (e.g. $\sA_0^{[0]} = \sA_0$, $\sA_0^{[1]} = \sA_0'$). Moreover, any subscheme will also be decorated with tick marks, so $\sZ'$ signifies that $\sZ' \subset \sA_0'$.
  
  As mentioned before, it will be necessary to observe how these subschemes of $\sA_0$ transform (either their strict transform or scheme-theoretic inverse image) in each step. To keep track of this, we will use a subscript to denote which step the subscheme will be used in. So for example, $\sC_4$ is a subscheme of $\sA_0$ and it will transform to $\sC_4'''$ in Step 4 which is the true center of the blowup $\sA_1^{[4]} \rightarrow \sA_1'''$.
  
  The subschemes that will be blown up arise from subschemes of $\sA_0$ that are the union of certain KR-strata. These strata are indexed by the $\mu$-admissible alcoves of $\Sp_4$. An alcove of $\Sp_4$ is an alcove of $\SL_4$ (see Section \ref{sect local model of sAGL_1}) that satisfies the following duality condition: there exists a $c \in \mathbb{Z}$ such that $x_i(j) - x_{4-i}(4-j+1) = c$ for $0 \leq i \leq 4$ and $1 \leq j \leq 4$. The affine Weyl group $W^\aff$ of $\Sp_4$ acts simply transitively on the alcoves of size $2$, and hence fixing the base alcove $\tau$ from Section \ref{Local model of sAGSp1} with $r = 2$ we can identify these two sets. Explicitly, $W^\aff$ is the subgroup of $\Z^{4} \rtimes S_4$ generated by the simple affine reflections $s_0 = (-1,0,0,1)(14)$, $s_1 = (12)(34)$, and $s_2 = (2,3)$. The alcoves that make up the $\mu$-admissible set are shown in Figure \ref{gsp4 mu-permissible set} in various shades of gray.  The dimension of the stratum corresponding to $x \in W^\aff$ is equal to $\ell(x)$, where $\ell(\cdot)$ is the length with respect to the Bruhat order. In particular the stratum $S_\tau$ corresponding to the base alcove $\tau$ is the unique KR-stratum of dimension zero. The irreducible components are the extreme alcoves which are shaded medium gray. The $p$-rank zero locus is pictured in dark gray, given by $\overline{\sA_{s_0s_2\tau} \cup \sA_{s_1\tau}}$. The stratification satisfies the property that for $x,y \in W^\aff$, $x \leq y$ with respect to the Bruhat order if and only if $S_x \subset \overline{S_y}$.
  
  \begin{figure}[t]
   \centering
   \includegraphics[scale=0.6]{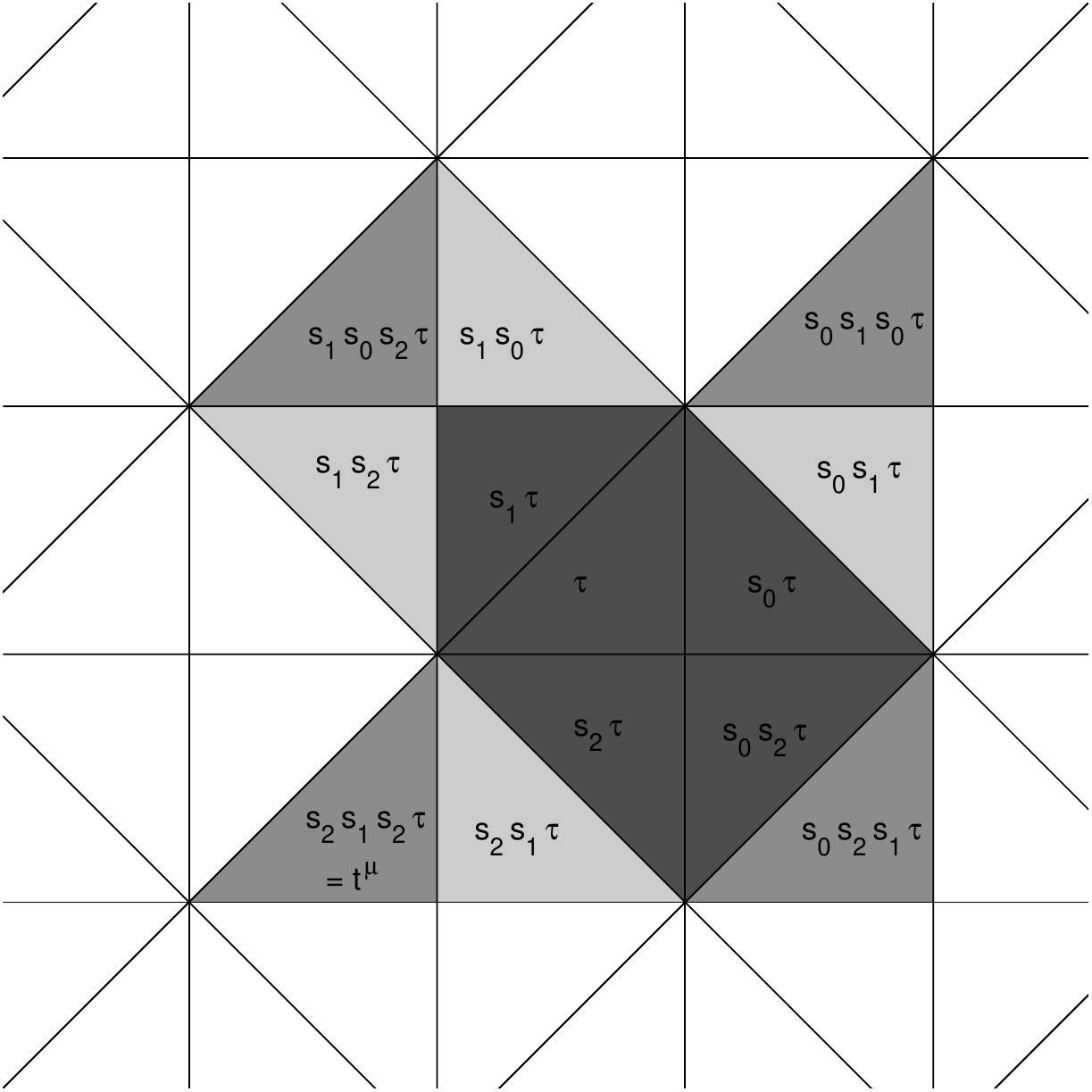}
   \caption{$\mu$-admissible set for $\GSp_4$}
   \label{gsp4 mu-permissible set}
  \end{figure}
  
  For $x \in \Adm(\mu)$, $\sA_x$ will denote the KR-stratum corresponding to $x$ and $n_x$ the number of connected components of $\sA_x$. The number $n_x$ is determined by the prime-to-$p$ level $K^p$ chosen \cite{GYu}.

  \subsection{Step 0: $U_0$, $U_1$, $\sA_0$, and $\sA_1$}
   \label{sect Step 0: U_0, U_1, sA_0, and sA_1}
   \subsubsection{Description of the local models $U_0$ and $U_1$}
    In Section \ref{Local model of sAGSp1}, $U_0$ is given as a subscheme of $\Spec(\sO_{\breve E}[a_{jk}^i; i = 0, \dots, 3, j,k = 1,2])$. By setting $x =  a_{22}^1$, $y = a_{11}^0$, $a = a_{12}^0$, $b = a_{12}^2$ and $c = -a_{12}^1$ we arrive at the \'{e}tale local model derived in \cite{dJ}:
     $$U_0 = \Spec(B) \quad \text{where} \quad B = \sO_{\breve E}[x,y,a,b,c]/(xy-p,ax+by+abc).$$
    With this presentation we have that, up to a unit, $q_0 = y$, $q_0^* = x$, $q_1 = y+ac$, and $q_1^* = x+bc$. The four irreducible components of $U_0 \otimes \overline{\F}_p$ are
     $$Z_{00} = Z(y,a), \quad Z_{01} = Z(y,x+bc), \quad Z_{10} = Z(x,y+ac), \quad Z_{11} = Z(x,b).$$
    Theorem \ref{main theorem for GSp with Iw_1(p) level} gives that $U_1$ is the spectrum of
     $$\sO_{\breve E}[x,y,a,b,c,u,v]/(xy-p, ax+by+abc,u^{p-1}-x,v^{p-1} - x -bc).$$
    The four irreducible components of $U_1 \otimes \overline{\F}_p$ correspond to the ideals $(u, v, b)$, $(u, y+ac)$, $(y, v)$, and $(y, a)$. We will use later that each irreducible component is normal, which can be seen by applying Serre's Criterion \cite[Theorem 23.8]{Matsumura}.
  
   \subsubsection{Description of the integral model $\sA_0$}
    As in \cite[Section 5]{dJ}, we define $\sZ_{ij} = \Phi(\Psi^{-1}(Z_{ij}))^\red$ where $\sA_0 \xleftarrow{\Phi} V \xrightarrow{\Psi} \Mloc$, and so $\sZ_{ij}$ \'{e}tale locally corresponds to $Z_{ij}$. These are precisely the four irreducible components of $\sA_0 \otimes \overline{\F}_p$ \cite[Theorem 1.1]{YuSiegel}. The following subschemes will be used throughout the construction.
    \begin{center}
     \begin{tabular}{rcl|rcl}
      \multicolumn{3}{c|}{Local model $U_0$} & \multicolumn{3}{c}{Integral model  $\sA_0$}\\
      \hline
      $Z_1$ &$=$& $Z(x,b)$ & $\sZ_1$ &$=$& $\sZ_{11}$\\
      $C_1$ &$=$& $Z(x,y,a,b)$ & $\sC_1$ &$=$& $\sZ_{00} \cap \sZ_{01} \cap  \sZ_{10} \cap \sZ_{11}$\\
      $Z_3$ &$=$& $Z(x,bc)$ & $\sZ_3$ &$=$& $\sZ_{11} \cup (\sZ_{01} \cap \sZ_{10})$\\
      $C_3$ &$=$& $Z(x,bc)$ & $\sC_3$ &$=$& $\sZ_{11} \cup (\sZ_{01} \cap \sZ_{10})$\\
      $Z_4$ &$=$& $Z(x,b)$ & $\sZ_4$ &$=$& $\sZ_{11}$\\
      $C_4$ &$=$& $Z(x,y,b,c)$ & $\sC_4$ &$=$& $\overline{(\sZ_{01} \cap \sZ_{10} \cap \sZ_{11}) \setminus \sZ_{00}}$
     \end{tabular}
     \captionof{table}{Subschemes of $U_0$ and $\sA_0$}
    \end{center}
    The subschemes on the left \'{e}tale locally correspond to those on the right.
    \begin{proposition}
     \label{subschemes of sA_0 as KR-strata}
     Writing each subscheme in the table above in terms of KR-strata, we have the following.
     \begin{center}
      \begin{tabular}{ccccl}
       $\sZ_{11}$ &=& $\overline{\sA_{s_2s_1s_2\tau}}$\\
       $\sZ_{01} \cap \sZ_{10}$ &=& $\overline{\sA_{s_0s_2\tau} \cup \sA_{s_1\tau}}$\\
       $\sZ_{00} \cap \sZ_{01} \cap \sZ_{10} \cap \sZ_{11}$ &=& $\overline{\sA_{s_1\tau}}$\\
       $\overline{(\sZ_{01} \cap \sZ_{10} \cap \sZ_{11}) \setminus \sZ_{00}}$ &=& $\overline{\sA_{s_2\tau}}$
      \end{tabular}
     \end{center}
    \end{proposition}
    \begin{proof}
     The locus of $U_0$ corresponding to the $p$-rank zero locus of $\sA_0$ is given by $Z(q_0,q_1,q_0^*,q_1^*) = Z_{01} \cap Z_{10}$. Hence $\sZ_{01} \cap \sZ_{10}$ is the $p$-rank zero locus of $\sA_0$ and this is given by $\overline{\sA_{s_0s_2\tau} \cup \sA_{s_1\tau}}$ \cite[Proposition 2.7]{GYu2}.
 
     To show $\sZ_{11} = \overline{\sA_{s_2s_1s_2\tau}}$, we choose the closed point $x = a = b = c = 0$ and $y = 1$ of $U_0$ which lies solely on the irreducible component $Z_{11}$ of the special fiber. This point corresponds to the flag
      $$\sF_0 = \left(\begin{smallmatrix}1 & 0\\0 & 1\\1 & 0\\0 &1\end{smallmatrix}\right), \quad \sF_1 = \left(\begin{smallmatrix}0 & 0\\1 & 0\\0 & 1\\1 & 0\end{smallmatrix}\right), \quad \sF_2 = \left(\begin{smallmatrix}0 & 0\\0 & 0\\1 & 0\\0 & 1\end{smallmatrix}\right), \quad \sF_3 = \left(\begin{smallmatrix}0 & 1\\0 & 0\\0 & 1\\1 & 0\end{smallmatrix}\right)$$
     and in the notation of \cite[Section 4]{GortzGL} (cf. \cite[Section 4.3]{GortzComputing}) this gives the associated lattice chain
      $$\left(\begin{smallmatrix}1 \\&\pi&&1\\1&&\pi\\&&&1\end{smallmatrix}\right), \quad \left(\begin{smallmatrix}1\\&1\\&&1\\&1&&\pi\end{smallmatrix}\right), \quad \left(\begin{smallmatrix}1\\&1\\&&1\\&&&1\end{smallmatrix}\right), \quad \left(\begin{smallmatrix}\pi^{-1}\\&1\\\pi^{-1}&&1\\&&&1\end{smallmatrix}\right).$$
     Our chosen point corresponds to closed points lying on $\sA_{s_2s_1s_2\tau}$ if and only if there is an element $b$ in the Iwahori subgroup such that $b \cdot s_2s_1s_2\tau$ gives the same lattice chain as above. With $s_2s_1s_2\tau$ corresponding to the lattice chain 
      $$\left(\begin{smallmatrix}\pi\\&\pi\\&&1\\&&&1\end{smallmatrix}\right), \quad \left(\begin{smallmatrix}1\\&\pi\\&&1\\&&&1\end{smallmatrix}\right), \quad \left(\begin{smallmatrix}1\\&1\\&&1\\&&&1\end{smallmatrix}\right), \quad \left(\begin{smallmatrix}1\\&1\\&&\pi^{-1}\\&&&1\end{smallmatrix}\right)$$
     it is easy to check that 
      $$b = \left(\begin{smallmatrix}1&&1\\&1&&1\\&&1\\&&&1\end{smallmatrix}\right)$$
     suffices. Therefore $\sZ_{11} = \overline{\sA_{s_2s_1s_2\tau}}$. Since $Z_{00} \cap Z_{11}$, and hence $\sZ_{00} \cap \sZ_{11}$, is the union of two dimensional and one dimensional components, from inspection of the $\mu$-admissible set it must be that $\sZ_{00} = \overline{\sA_{s_0s_1s_0\tau}}$. That $\sZ_{00} \cap \sZ_{01} \cap \sZ_{10} \cap \sZ_{11}=\sA_\tau \cup \sA_{s_1\tau}$ is now immediate. Finally $\sZ_{01} \cap \sZ_{10} \cap \sZ_{11} = \sA_\tau \cup \sA_{s_1\tau} \cup \sA_{s_2\tau}$ and thus we have $(\sZ_{01} \cap \sZ_{10} \cap \sZ_{11}) \setminus \sZ_{00} = \sA_{s_2\tau}$.
    \end{proof}
    \begin{proposition}
     \label{Connected components A_0}
     The number of connected and irreducible components of the subschemes of $\sA_0$ are as follows.
     \begin{center}
      \begin{tabular}{c|c|c}
       Subscheme of $\sA_0$ & \# connected & \# irreducible\\
       \hline
       $\sZ_1$ & 1 & 1\\
       $\sC_1$ & $n_{s_1\tau}$ & $n_{s_1\tau}$\\
       $\sZ_3$ & 1 & $1+n_{s_0s_2\tau}$\\
       $\sC_3$ & 1 & $1+n_{s_0s_2\tau}$\\
       $\sZ_4$ & 1 & 1\\
       $\sC_4$ & $n_{s_2\tau}$ & $n_{s_2\tau}$\\
      \end{tabular}
      \captionof{table}{Subschemes of $\sA_0$}
     \end{center}
    \end{proposition}
    \begin{proof}
     To show that $\sZ_3 = \sZ_{11} \cup (\sZ_{01} \cap \sZ_{10})$ is connected, it suffices to show that each connected component of $\sZ_{01} \cap \sZ_{10}$ meets $\sZ_{11}$. Let $\sW$ be such a connected component. Since $\sZ_{01} \cap \sZ_{10}$ is a union of KR-strata, by possibly shrinking $\sW$ we may assume $\sW$ is a connected component of some KR-stratum. Then $\overline{\sW} \cap \sA_{\tau} \neq \emptyset$ \cite[Theorem 6.4]{GYu2}, where $\overline{\sW}$ is the Zariski closure of $\sW$ inside of $\sA_0$. As $\sA_{\tau} \subset \sZ_{11}$, the claim follows.

     From Proposition \ref{subschemes of sA_0 as KR-strata}, $\sZ_{11} \cup (\sZ_{01} \cap \sZ_{10})$ is a union of three and two dimensional irreducible components. The unique three dimensional component is $\sZ_{11}$ and the two dimensional components are given by the irreducible components of $\overline{(\sZ_{01} \cap \sZ_{10}) \setminus \sZ_{11}} = \overline{\sA_{s_0s_2\tau}}$. As $\overline{(Z_{01} \cap Z_{10}) \setminus Z_{11}}$ corresponds to the ideal $(x,y,c)$ of $B$, we have that $\overline{\sZ_{01} \cap \sZ_{10} \setminus \sZ_{11}}$ is smooth. Thus each of the $n_{s_0s_2\tau}$ connected components of $\overline{\sA_{s_0s_2\tau}}$ is irreducible. The result now follows.
     
     The rest of the proposition is clear.
    \end{proof}
   
   \subsubsection{Description of the integral model $\sA_1$}
    \begin{proposition}
     \label{IrredInA1}
     $\sA_1 \otimes \overline{\F}_p$ is connected, equidimensional of dimension three, and the irreducible components are normal. Furthermore, $\sA_1 \otimes \overline{\F}_p$ and has precisely four irreducible components given by $\pi^{-1}(\sZ_{ij})^\red$, where $\pi : \sA_1 \rightarrow \sA_0$.
    \end{proposition}
    \begin{proof}
     That $\sA_1 \otimes \overline{\F}_p$ is equidimensional of dimension three is immediate from inspection of the local model $U_1 \otimes \overline{\F}_p$. Let $x \in \sA_\tau \subset \sA_0$ be any closed point and consider $\pi^{-1}(x)$. With $\pi$ finite and surjective, each irreducible component of $\sA_1 \otimes \overline{\F}_p$ maps surjectively onto an irreducible component of $\sA_0 \otimes \overline{\F}_p$. Furthermore, every irreducible component of $\sA_0 \otimes \overline{\F}_p$ contains $\sA_\tau$ \cite[Theorem 6.4]{GYu2}. Thus $\pi^{-1}(x)$ meets every irreducible component of $\sA_1 \otimes \overline{\F}_p$. But $\sA_\tau$ is in the $p$-rank zero locus of $\sA_0$ (see Figure \ref{gsp4 mu-permissible set} above), and so $\pi^{-1}(x)$ consists of a single closed point. Therefore $\sA_1 \otimes \overline{\F}_p$ is connected.
     
     Let $\sA_1 \xleftarrow{\varphi} V \xrightarrow{\psi} U_1$ be an \'{e}tale cover of $\sA_1$ with \'{e}tale morphism to $U_1$. Choose a closed point $v$ of $V$ such that $\varphi(v) = \pi^{-1}(x)$ and set $y = \psi(v)$. Since the irreducible components of $U_1 \otimes \overline{\F}_p$ are integral, normal, and excellent, the completion at $y$ of any irreducible component that $y$ lies on is also integral and normal. Therefore there are at most four irreducible components of $V$ passing through $v$. Thus the number of irreducible components of $\sA_1 \otimes \overline{\F}_p$ passing through $\pi^{-1}(x)$ is at most four, and all irreducible components of $\sA_1 \otimes \overline{\F}_p$ pass through $\pi^{-1}(x)$. Since $\sA_1 \otimes \overline{\F}_p \rightarrow \sA_0 \otimes \overline{\F}_p$ is surjective, there must be precisely four. Therefore they are given by $\pi^{-1}(\sZ_{ij})^\red$.
    \end{proof}

  \subsection{Step I: Semi-stable resolution of $\sA_0$ \cite{dJTalk}}
   Set $\sA_0' = \Bl_{\sZ_1}(\sA_0)$ and $U_0' = \Bl_{Z_1}(U_0)$.
   \subsubsection{Description of $U_0'$}
    \label{Step I description of local model}
    An easy calculation shows that $U_0' = \Proj_{U_0}\left(B[\widetilde{x},\widetilde{b}]/(a\widetilde{x} + \widetilde{b}y+a\widetilde{b}c, x\widetilde{b}-\widetilde{x}b)\right)$
    where $\widetilde{x}$ and $\widetilde{b}$ are of grade one. The true center of $\rho_U' : U_0' \rightarrow U_0$ is given by $C_1 = Z(x,y,a,b)$ and is of dimension one, the true center of $\rho_U'$ is equal to its fundamental center, the exceptional locus of $\rho_U'$ is two dimensional, and $U_0' \otimes \overline{\F}_p$ is equidimensional of dimension three. The strict transforms of the closed subschemes given in Step 0 are
     $$Z_{3}' = C_{3}' = Z(x,b) \cup Z(x,y,c,\widetilde{x}), \quad Z_4' = Z(x,b), \quad C_4' = Z(x,y,b,c).$$
    \vspace{-\baselineskip}\vspace{-\baselineskip}

   \subsubsection{Description of $\sA_0'$}
    With $U_0'$ an \'{e}tale local model of $\sA_0'$, the true center of $\sA_0' \rightarrow \sA_0$ is $\sC_1$ since it \'{e}tale locally corresponds to $C_1$. By the remarks in the previous section, $\sA_0' \otimes \overline{\F}_p$ is equidimensional of dimension three and the exceptional locus of $\sA_0' \rightarrow \sA_0$ is two dimensional. Thus no irreducible component of $\sA_0' \otimes \overline{\F}_p$ is contained in the exceptional locus. Therefore $\sA_0' \otimes \overline{\F}_p$ has four irreducible components, each being given by the strict transform of an irreducible component of $\sA_0 \otimes \overline{\F}_p$. We denote the strict transform of $\sZ_{ij}$ by $\sZ_{ij}'$.

    \begin{proposition}
     \label{Connected components A_0'}
     The number of connected and irreducible components of the subschemes of $\sA_0'$ are as follows.
     \begin{center}
      \begin{tabular}{c|c|c}
       Subscheme of $\sA_0'$ & \# connected & \# irreducible\\
       \hline
       $\sZ_3' = \ST(\sZ_3)$ & 1 & $n_{s_0s_2\tau}+1$\\
       $\sC_3' = \ST(\sC_3)$ & 1 & $n_{s_0s_2\tau}+1$\\
       $\sZ_4' = \ST(\sZ_4)$ & 1 & $1$\\
       $\sC_4' = \ST(\sC_4)$ & $n_{s_2\tau}$ & $n_{s_2\tau}$\\
      \end{tabular}
      \captionof{table}{Subschemes of $\sA_0'$}
     \end{center}
    \end{proposition}
    \begin{proof}
     We start by showing $\sZ_3'$ is connected. By Proposition \ref{subschemes of sA_0 as KR-strata} and inspection of Figure \ref{gsp4 mu-permissible set}, $\sZ_3 = \sZ_{11} \cup (\sZ_{01} \cap \sZ_{10})$ is a union of three and two dimensional components intersecting in a one dimensional closed subscheme. Set $\sW = \overline{\sZ_{01} \cap \sZ_{10} \setminus \sZ_{11}}$, which is equidimensional of dimension two. Then $\sZ_3 = \sZ_{11} \cup \sW$ and from Proposition \ref{subschemes of sA_0 as KR-strata}, we can express the subschemes as the union of KR-strata as follows.
     \begin{align*}
      \sZ_{11} &= \sA_\tau \cup \sA_{s_1\tau} \cup \sA_{s_2\tau} \cup  \sA_{s_2s_1\tau} \cup \sA_{s_2s_1s_2\tau}\\
      \sW &= \sA_\tau \cup \sA_{s_0\tau} \cup \sA_{s_2\tau} \cup \sA_{s_0s_2\tau}\\
      \sC_1 &= \sA_\tau \cup \sA_{s_1\tau}
     \end{align*}
     Therefore the one dimensional subscheme $\sZ_{11} \cap \sW$ intersects with $\sC_1$ in $\sA_\tau$, a zero dimensional subscheme. With $\sZ_{11}$ and $\sW$ smooth, $\sZ_3 \setminus \sC_1$ is connected, and thus so is the strict transform of $\sZ_3$. Therefore $\sZ_3'$ is connected.

     The irreducible components of $\sZ_3$, $\sZ_4$, and $\sC_4$ are not contained in $\sC_1$, and so each of these three subschemes has the same number of irreducible components as their strict transform. With $\sZ_4$ of dimension 3, $\sZ_4 \setminus C_1$ is connected and hence so is $\sZ_4'$. Since $\sC_4$ is smooth, it follows that $\sC_4$ and $\sC_4'$ have the same number of connected components. With $\sC_3' = \sZ_3'$, all the claims have been shown.
    \end{proof}

  \subsection{Step II: Fiber with $\sA_1$}
   Set $\sA_1'' = \sA_1 \times_{\sA_0} \sA_0'$ and $U_1'' = U_1 \times_{U_0} U_0'$.
   \subsubsection{Description of $U_1''$}
    As in Theorem \ref{main theorem for GSp with Iw_1(p) level}, $U_1$ is given by adjoining the variables $u$ and $v$ along with the relations $u^{p-1} - x$ and $v^{p-1} - (x+bc)$ to $U_0$. Thus we have
     $$U_1'' = \Proj\left(B[u,v][\widetilde{x},\widetilde{b}]/(a\widetilde{x} + \widetilde{b}y+a\widetilde{b}c, x\widetilde{b}-\widetilde{x}b, u^{p-1} - x, v^{p-1}-(x+bc)\right)$$
    where $B[u,v]$ is of grade 0 and $\widetilde{x}$ and $\widetilde{b}$ are of grade 1.  The reduced inverse images under $U_1'' \rightarrow U_0'$ are give by
     $$Z_{3}'' = C_{3}'' = Z(u,v,b) \cup Z(u,v,c,\widetilde{x}), \quad Z_{4}'' = Z(u,v,b), \quad C_{4}'' = Z(u,v,b,c,\widetilde{x}).$$
    Note that $Z(u, v, b) \cup Z(u, v, c, \widetilde{x}) = Z(u, v)$ as the relation $v^{p-1}-u^{p-1}-bc$ implies that if $u$ and $v$ are zero, then $bc = 0$ giving the two components.
   \subsubsection{Description of $\sA_1''$}
    With $\sA_1'' = \sA_1 \times_{\sA_0} \sA_0'$, the projection $\sA_1'' \rightarrow \sA_1$ is proper and birational and so it is a modification. The projection $\rho_\sA'' : \sA_1'' \rightarrow \sA_0'$ is finite and flat. 
    \begin{proposition}
     Set $\sZ_{ij}'' = (\rho_\sA'')^{-1}(\sZ_{ij}')^\red$. Each $\sZ_{ij}''$ is an irreducible component of $\sA_1'' \otimes \overline{\F}_p$, and these give all the irreducible components of $\sA_1'' \otimes \overline{\F}_p$.
    \end{proposition}
    \begin{proof}
     From Proposition \ref{IrredInA1} we have that $\sW_{ij} = \pi^{-1}(\sZ_{ij})^\red$ is irreducible, where $\pi : \sA_1 \rightarrow \sA_0$. The modification $\sA_1'' \rightarrow \sA_1$ has true center of dimension at most one, and thus $\sW_{ij}$ is not contained in the true center. Therefore its strict transform $\sW_{ij}''$ with respect to $\sA_1'' \rightarrow \sA_1$ is irreducible.
 
     Set $\sU = \sA_0 \setminus \sC_1$, $\sU' = (\rho_\sA')^{-1}(\sU)$, and $\sU'' = (\rho_\sA'')^{-1}(\sU')$. Then $\sZ_{ij}'' \cap \sU'' = \sW_{ij}'' \cap \sU''$ because both can be described as the reduced inverse image of $\sZ_{ij} \cap \sU$ under the two paths in the following cartesian diagram.
     \begin{equation*}
      \begin{tikzcd}
       \sA_1'' \arrow{r}{\rho_\sA''} \arrow{d} &
       \sA_0' \arrow{d}{\rho_\sA'} \\
       \sA_1 \arrow{r}{\pi} &
       \sA_0
      \end{tikzcd}
     \end{equation*}
     Noting that $\sZ_{ij}'$ is irreducible making $\sZ_{ij}' = \overline{\sZ_{ij}' \cap \sU'}$, as sets we have
     \begin{align*}
     	\sZ_{ij}'' &= (\rho_\sA'')^{-1}(\overline{\sZ_{ij}' \cap \sU'}) = \overline{(\rho_\sA'')^{-1}(\sZ_{ij}' \cap \sU')} = \overline{(\rho_\sA'')^{-1}(\sZ_{ij}') \cap (\rho_\sA'')^{-1}(\sU')}\\
     	&= \overline{\sZ_{ij}'' \cap \sU''}
     \end{align*}
     where the second equality follows since $\rho_\sA''$ is flat. It thus suffices to show that $\sZ_{ij}'' \cap \sU''$ is irreducible. But this is immediate since $\sZ_{ij}'' \cap \sU'' = \sW_{ij}'' \cap \sU''$ with $\sW_{ij}''$ irreducible. That the collection $\set{\sZ_{ij}''}$ gives all the irreducible components is immediate.
    \end{proof}

    \begin{proposition}
     \label{Connected components for sA_1''}
     The number of connected and irreducible components of the subschemes of $\sA_1''$ are as follows.
     \begin{center}
      \begin{tabular}{c|c|c}
       Subscheme of $\sA_1''$ & \# connected & \# irreducible\\
       \hline
       $\sZ_3'' = (\rho_\sA'')^{-1}(\sZ_3')^\red$ & 1 & $n_{s_0s_2\tau}+1$\\
       $\sC_3'' = (\rho_\sA'')^{-1}(\sC_3')^\red$ & 1 & $n_{s_0s_2\tau}+1$\\
       $\sZ_4'' = (\rho_\sA'')^{-1}(\sZ_4')^\red$ & 1 & $1$\\
       $\sC_4'' = (\rho_\sA'')^{-1}(\sC_4')^\red$ & $n_{s_2\tau}$ & $n_{s_2\tau}$\\
      \end{tabular}
      \captionof{table}{Subschemes of $\sA_1''$}
     \end{center}
    \end{proposition}
    \begin{proof}
     The reader may wish to consult the diagram in the previous proof. From the proof of Proposition \ref{Connected components A_0'}, $\sW'$ arises as the strict transform of an irreducible component of $\sZ_3$. If $\sW'$ is the strict transform of $\sZ_{11}$ then the claim is clear, so assume that $\sW'$ is the strict transform of some two dimensional irreducible component of $\sZ_3$. From the proof of Proposition \ref{Connected components A_0}, it must be that this two dimensional component of $\sZ_3$ is contained in $\sZ_{01} \cap \sZ_{10}$ and hence $\rho_\sA'(\sW') \subset \sZ_{01} \cap \sZ_{10}$. Let $x'$ be a closed point of $\sW'$ and set $x = \rho_\sA'(x')$. Then $x$ is in $\sZ_{01} \cap \sZ_{10}$, the $p$-rank zero locus, and so $\pi^{-1}(x)$ consists of a single closed point. Therefore the fiber $(\rho_\sA'')^{-1}(x')$ also consists of a single closed point and so $(\rho_\sA'')^{-1}(\sW')$ is irreducible. From this argument it also follows that $\sZ_3''$ is connected.
 
     The statements about $\sC_4''$ follow in a similar manner to those about $\sZ_3''$, $\sZ_4'' = \sZ_{11}''$, and $\sC_3'' = \sZ_3''$.
    \end{proof}
  \subsection{Step III: Blowup $\sZ_3''$.}
   Set $\sA_1''' = \Bl_{\sZ_3''}(\sA_1'')$ and $U_1''' = \Bl_{Z_3''}(U_1'')$.
   \subsubsection{Description of $U_1'''$}
    \label{Description of the local model U_1'''}
    To simplify the notation, let $X = U_1'''$. A presentation of $X$ is given by the closed subscheme of $\Proj_{U_1''}(\sO[\widetilde{u},\widetilde{v}])$, where $\widetilde{u}$ and $\widetilde{v}$ are of grade 1, cut out by the following equations.
     $$u\widetilde{v} - \widetilde{u}v, \quad (\widetilde{x}+\widetilde{b}c)\widetilde{u}^{p-1}-\widetilde{x}\widetilde{v}^{p-1}, \quad y\widetilde{u}^{p-1}-(y+ac)\widetilde{v}^{p-1}$$
    The scheme $X$ is covered by four standard affine charts.
    \begin{center} 
     \begin{tabular}{ccl}
      $X_{00}$ & $\widetilde{x} = \widetilde{u} = 1$ & $\sO_{\breve E}[y,a,c,u,\widetilde{b},\widetilde{v}]/(u^{p-1}y-p, \widetilde{v}^{p-1}-(1+\widetilde{b}c))$\\
      $X_{01}$ & $\widetilde{x} = \widetilde{v} = 1$ & $\sO_{\breve E}[y,c,v,\widetilde{b},\widetilde{u}]/(v^{p-1}\widetilde{u}^{p-1}y-p,\widetilde{u}^{p-1}(1+\widetilde{b}c)-1)$\\
      $X_{10}$ & $\widetilde{b} = \widetilde{u} = 1$ & $\sO_{\breve E}[a,b,u,\widetilde{x},\widetilde{v}]/(b\widetilde{x}^2\widetilde{v}^{p-1}(-a)-p, u^{p-1}-b\widetilde{x})$\\
      $X_{11}$ & $\widetilde{b} = \widetilde{v} = 1$ & $\sO_{\breve E}[a,b,s,v,\widetilde{u}]/(bs^2\widetilde{u}^{p-1}(-a)-p, v^{p-1}-bs)$
     \end{tabular}
    \end{center}
    The last chart uses a change of coordinates $s=\widetilde{x}+c$. We can cover $X_{00}$ with two open subschemes each defined respectively by the condition $\widetilde{b}$ and $1+\widetilde{b}c$ is invertible. These open subschemes are
    \begin{center} 
     \begin{tabular}{ccl}
      $X_{00}'$ & $\widetilde{b} \neq 0$ & $\sO_{\breve E}[a,u,\widetilde{b}^{\pm 1},\widetilde{v}]/(u^{p-1}(-a)\widetilde{v}^{p-1}\widetilde{b}^{-1}-p)$\\
      $X_{00}''$ & $1+\widetilde{b}c \neq 0$ & $\sO_{\breve E}[y,c,u,\widetilde{b},\widetilde{v},(1+\widetilde{b}c)^{-1}]/(u^{p-1}y-p, \widetilde{v}^{p-1}-(1+\widetilde{b}c)).$
     \end{tabular}
    \end{center}
    Since $X_{00}' \subset X_{10}$ as an open subscheme, $X$ is covered by $X_{00}''$, $X_{01}$, $X_{10}$, and $X_{11}$. By inspection we see that $X = U_1'''$ is normal, the true center of $\rho_U''' : U_1''' \rightarrow U_1''$ is $C_3'' = Z(u,v)$, the fundamental center is $Z_{01}'' \cap Z_{10}'' = Z(u,v,\widetilde{x},c)$, and the residual locus is $Z_{11}'' \setminus (Z_{01}'' \cap Z_{10}'')$. The fundamental center is two dimensional and smooth, the residual locus is three dimensional, and the exceptional locus is equidimensional of dimension three. Note that $C_4'' = Z(u,v,b,c,\widetilde{x})$ is smooth of dimension two and $C_4''$ intersects with the fundamental center of $U_1''' \rightarrow U_1''$ in a smooth one dimensional subscheme. Taking the strict transform under $\rho_U'''$ we have $Z_{4}''' = Z(u,v,b)$ and $C_{4}''' = Z(u,v,b,c,\widetilde{x})$.
   \subsubsection{Description of $\sA_1'''$}
    The integral model $\sA_1'''$ has $U_1'''$ as an \'{e}tale local model, so $\sA_1'''$ is normal.

    \begin{proposition}
     The special fiber of $\sA_1'''$ has precisely $4+n_{s_0s_2\tau}$ irreducible components. Three are given by the strict transforms of $\sZ_{00}''$, $\sZ_{01}''$, and $\sZ_{10}''$. The other $1+n_{s_0s_2\tau}$ are contained in the exceptional locus: one lying above $\sZ_{11}''$ and one lying above each two dimensional irreducible component of $\sZ_3''$.
    \end{proposition}
    \begin{proof}
     As can be seen from the local model, the exceptional locus $\sE'''$ of $\sA_1''' \rightarrow \sA_1''$ is equidimensional of dimension three. As $\sC_3''$ has $1+n_{s_0s_2\tau}$ irreducible components by Proposition \ref{Connected components for sA_1''}, $\sE'''$ must consist of at least $1+n_{s_0s_2\tau}$ irreducible components. Denote the irreducible components of $\sE'''$ by $\set{\sW_i'''}$. Without loss of generality assume that $\rho_\sA'''(\sW_1''') \subset \sZ_{11}''$ and that $\rho_\sA'''(\sW_2'''), \dotsc, \rho_\sA'''(\sW_{1+n_{s_0s_2\tau}}''')$ are each contained in a distinct two dimensional irreducible component of $\sZ_3''$.

     For any $i$, if $\rho_\sA'''(\sW_i''') \subset \sZ_{11}''$ then $\rho_\sA'''(\sW_i''') = \sZ_{11}''$. Indeed, since $\rho_\sA'''$ is proper it suffices to show that $\rho_\sA'''(\sW_i''')$ is three dimensional. As can be seen from the local model, the fiber above a closed point of $\sA_1''$ is at most one dimensional. Hence if $\rho_\sA'''$ maps $\sW_i'''$ to something of dimension smaller than three, then the fibers of $\rho_\sA'''$ must be of dimension at least one on a two dimensional subscheme of $\sZ_{11}''$. As this does not occur on the local model, it can not occur here as well.

     Now consider a closed point $x''$ of $\overline{\sC_3'' \setminus \sZ_{11}''}$, the fundamental center of $\sA_1''' \rightarrow \sA_1''$. From the local model, the fiber above $x''$ with respect to $\rho_\sA'''$ is connected and smooth. Thus for each irreducible component of $\overline{\sC_3'' \setminus \sZ_{11}''}$, there is a single irreducible component of the exceptional locus of $\rho_\sA'''$ mapping surjectively onto it. Recall we have labeled these components $\sW_2''', \dotsc, \sW_{1+n_{s_0s_2\tau}}'''$.

     To show $(\rho_\sA''')^{-1}(\sZ_{11}'')^\red = \sW_1'''$, it suffices to show that $(\rho_\sA''')^{-1}(\sZ_{11}'')^\red$ is irreducible. From the local model, $(\rho_\sA''')^{-1}(\sZ_{11}'')^\red$ is smooth and equidimensional of dimension three. This implies that $(\rho_\sA''')^{-1}(\sZ_{11}'')^\red$ is a disjoint union of irreducible components of $\sA_1''' \otimes \overline{\F}_p$. As each of these irreducible components maps into $\sZ_{11}''$, by the above we have that they map surjectively onto $\sZ_{11}''$. Thus there is a closed point $x''$ of $\sZ_{11}'' \cap \overline{\sC_3'' \setminus \sZ_{11}''}$ contained in the image of every irreducible component. As remarked before, the fiber above $x''$ is connected, and hence $(\rho_\sA''')^{-1}(\sZ_{11}'')^\red$ is connected. The claim now follows.

     Suppose there exists another irreducible component $\sW'''$. By the above, it must be that $\sW'''$ is mapped via $\rho_\sA'''$ into $\sZ_{11}''$. But then $\sW'''$ is contained in $(\rho_\sA''')^{-1}(\sZ_{11}'')^\red = \sW_1'''$ and therefore $\sW''' = \sW_1'''$.
    \end{proof}

    \begin{proposition}\label{prop connected and irred components of sA_1'''}
     The number of connected and irreducible components of the subschemes of $\sA_1'''$ are as follows.
     \begin{center}
      \begin{tabular}{c|c|c}
       Subscheme of $\sA_1'''$ & \# connected & \# irreducible\\
       \hline
       $\sZ_4''' = \ST(\sZ_{11}'')$ & $1$ & $1$\\
       $\sC_4''' = \ST(\sC_4'')$ & $n_{s_2\tau}$ & $n_{s_2\tau}$
      \end{tabular}
      \captionof{table}{Subschemes of $\sA_1'''$}
    \end{center}
    \end{proposition}
    \begin{proof}
     The statements about $\sZ_4'''$ were proved in the previous proposition. From the local models, $\sC_4''$ and $\sC_4'''$ are both smooth so each connected component is irreducible. It thus suffices to show that for a connected component $\sW'' \subset \sC_4''$, $\ST(\sW'')$ is connected. From the local model, the fiber above every closed point of $\sC_4''$ is connected, and hence $\ST(\sW'')$ is connected as well.
    \end{proof}

  \subsection{Step IV: Successively blowing up $\sZ_{11}'''$ and its strict transforms.}
   In this last step we define $\sA_1^{[i]}$ for $4 \leq i \leq p+1$ by first blowing up $\sZ_{4}'''$ in $\sA_1'''$, and then blowing up the strict transform of $\sZ_{4}'''$ in each successive modification. Likewise, we define $U_1^{[i]}$ for $4 \leq i \leq p+1$ by blowing up $Z_4'''$ in $U_1'''$ and then blowing up the strict transform of $Z_4'''$ in each successive modification.
   \subsubsection{Description of $U_1^{[i]}$, $3 \leq i \leq p+1$}
    From Section \ref{Description of the local model U_1'''}, $Z_4''' = Z(u,v,b)$. This subscheme is Cartier except on the  charts $X_{10}$ and $X_{11}$, and both of these charts have a presentation 
     $$Y = \Spec(A), \quad A = \sO_{\breve E}[x_1,x_2,x_3,x_4,u]/(x_1x_2^2x_3^{p-1}x_4 - p, u^{p-1} - x_1x_2)$$
    in which $Z_3'''$ is given by the subscheme $W = Z(u,x_1)$. Set $Y^{[0]} = Y$, $W_0 = W$, and for $1 \leq i \leq p-2$ define $Y^{[i]}$ inside $Y \times \overbrace{\mathbb{P}^1 \times \dots \times \mathbb{P}^1}^{i-\text{times}}$ by 
      $$u^{[i]}u^{p-i-1} - x_1^{[i]} x_2, \qquad ux_1^{[1]} - x_1u^{[1]},$$
      $$uu^{[j-1]}x_1^{[j]} - x_1^{[j-1]}u^{[j]} \quad \text{ for } \quad 2 \leq j \leq i$$
    where we are writing $x_1^{[0]} = x_1$ and $x_1^{[i]}$ for projective coordinates. Let $W_{i}$ be the strict transform of $W_{i-1}$ in $Y^{[i]}$ for each $i \geq 1$.
    
    \begin{proposition}
     \label{Blowup}
     For $1 \leq i \leq p-2$ we have the following.
     \begin{compactenum}[(i)]
      \item $Y^{[i]} \cong \Bl_{W_{i-1}}(Y^{[i-1]})$.
      \item The true center of $Y^{[i]} \rightarrow Y^{[i-1]}$ is one dimensional and smooth.
      \item The fundamental center of $Y^{[i]} \rightarrow Y^{[i-1]}$ is equal to its true center.
      \item The exceptional locus of $Y^{[i]} \rightarrow Y^{[i-1]}$ is smooth and two dimensional.
     \end{compactenum}
     Furthermore, $Y^{[p-2]}$ is regular with special fiber a divisor with normal crossings.
    \end{proposition}
    \begin{proof}
     (i) We proceed by induction. So assume $W_{i-1}$ corresponds to the ideal $(u, x_1^{[i-1]})$ which is certainly true for $i = 1$. The standard affine charts of $Y^{[i]}$, indexed by $1 \leq k \leq i+1$, are described by the conditions 
      $$u^{[j]} \neq 0 \quad \text{for} \quad 1 \leq j < k \qquad \text{and} \qquad x_1^{[j]} \neq 0 \quad \text{for} \quad k \leq j \leq i.$$
     In order to explicitly write them we must consider three cases.

     Case $k=1$: The coordinate ring is 
      $$\sO_{\breve E}\left[x_1, \frac{u^{[1]}}{x_1^{[1]}}, x_3, x_4\right]/\left(x_1^{2p-3}\left(\frac{u^{[1]}}{x_1^{[1]}}\right)^{2p-2}x_3^{p-1}x_4 - p\right).$$

     Case $1 < k < i+1$: The coordinate ring is 
      $$\sO_{\breve E}\left[\frac{x_1^{[k-1]}}{u^{[k-1]}}, \frac{u^{[k]}}{x_1^{[k]}}, x_3, x_4\right]/\left(\left(\frac{x_1^{[k-1]}}{u^{[k-1]}}\right)^{2p-k-2}\left(\frac{u^{[k]}}{x_1^{[k]}}\right)^{2p-k-1}x_3^{p-1}x_4 - p\right).$$

     Case $k=i+1$: The coordinate ring is 
      $$\sO_{\breve E}\left[\frac{x_1^{[i]}}{u^{[i]}}, x_2, , x_3, x_4, u\right]/\left(u^i\frac{x_1^{[i]}}{u^{[i]}}x_2^2x_3^{p-1}x_4 - p, u^{p-i-1} - \frac{x_1^{[i]}}{u^{[i]}}x_2\right).$$

     Note that each chart is integral. Direct computation shows that the equations defining $Y^{[i]}$ are part of those defining $\Bl_{W_{i-1}}(Y^{[i-1]})$, so there is a closed immersion $\iota : \Bl_{W_{i-1}}(Y^{[i-1]}) \rightarrow Y^{[i]}$ which is an isomorphism on the generic fiber. With $Y^{[i]}$ integral and of the same dimension as $\Bl_{W_{i-1}}(Y^{[i-1]})$, this implies $\iota$ is an isomorphism.

     To complete the induction, we must show that the strict transform of the subscheme $Z(u, x_1^{[i-1]})$ of $Y^{[i-1]}$ is the subscheme $Z(u, x_1^{[i]})$ of $Y^{[i]}$. From the charts above, the true center of $Y^{[i]} \rightarrow Y^{[i-1]}$ is given by $Z(u, x_1^{[i-1]}, x_2)$. Thus outside of the true center we have that $x_2$ is invertible, and so from the relation $u^{[i]}u^{p-i-1} - x_1^{[i]} x_2$ of $Y^{[i]}$ we get that $x_1^{[i]}$ is in the ideal defining the strict transform. As subschemes of $Y^{[i]}$, $Z(u, x_1^{[i-1]}, x_1^{[i]}) = Z(u, x_1^{[i]})$. This subscheme is irreducible and of dimension three and therefore we conclude it must be the strict transform of $W_{i-1}$.

     The remainder of the proposition now follows from inspection of the above charts.
    \end{proof}
    Using the explicit equations above, we record the global structure of the irreducible components of the special fiber.
    \begin{lemma}
     \label{tU_1 irreducible components}
     The irreducible components of $U_1^{[p-2]} \otimes \overline{\F}_p$ are described as follows.
     \begin{compactitem}
      \item There are $p+3$ components. Three components are given by $Z(\widetilde{u})$, $Z(\widetilde{v})$ and $Z(a)$. We index the other components by $1 \leq i \leq p$. For $1 \leq i \leq p-1$, the $i$th irreducible component is given by
       $$Z_i = Z(u, b, \widetilde{x}, b^{[1]}, b^{[2]}, \dotsc, b^{[i-2]}, u^{[i]}, u^{[i+1]}, \dotsc, u^{[p-2]})$$
      and the $p$th irreducible component is given by
       $$Z_p = Z(u, b, b^{[1]}, b^{[2]}, \dotsc, b^{[p-2]}).$$
      \item The components $Z(\widetilde{u})$, $Z(\widetilde{v})$, $Z(a)$, and $Z_i$ have multiplicity $p-1$, $p-1$, $1$, and $2p-i-1$ respectively. In particular, $Z_{p-1}$ is the only component with multiplicity divisible by $p$.
      \item The components $Z_1$ and $Z_p$ are isomorphic to $\A^{3}_{\overline{\F}_p}$. The components $Z_i$ with $2 \leq i \leq p-1$ are isomorphic to $\mathbb{P}^{1}_{\overline{\F}_p} \times \A^{2}_{\overline{\F}_p}$.
      \item The components intersect as indicated in the following ``dual complex'', drawn for $p = 5$. Each vertex represents an irreducible component and the label indicates its multiplicity. Each edge indicates that the two irreducible components intersect.
      \begin{center}
       \includegraphics[scale=0.3]{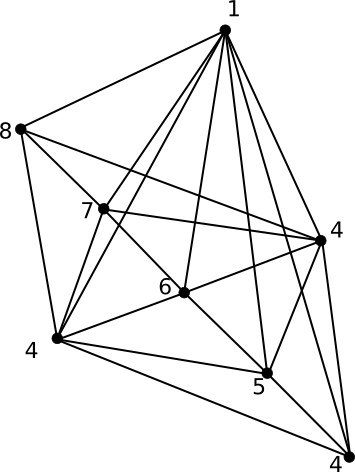}
       \captionof{figure}{Dual complex for $U_1^{[p-2]} \otimes \overline{\F}_p$ where $p=5$}
      \end{center}
      Moreover, a full subgraph that is a $k$-simplex indicates a $(k+1)$-fold intersection of the corresponding irreducible components, and conversely any $(k+1)$-fold intersection corresponds to a full subgraph which is a $k$-simplex.
      \item An intersection of components as indicated by a $k$-simplex has dimension $3-k$ over $\Spec(\overline{\F}_p)$.
     \end{compactitem}
    \end{lemma}

   \subsubsection{Description of $\sA_1^{[i]}$, $3 \leq i \leq p+1$}
    \begin{proposition}For $3 \leq i \leq p+1$, the scheme $\sA_1^{[i]} \otimes \overline{\F}_p$ has $4 + n_{s_0s_2\tau} + (i-3) \cdot n_{s_2\tau}$ irreducible components.
    \end{proposition}
    \begin{proof}
     We recall the following facts from Section \ref{Description of the local model U_1'''}, Proposition \ref{prop connected and irred components of sA_1'''}, and Proposition \ref{Blowup}.
     \begin{compactenum}[(i)]
      \item $\sC_4^{[3]}$ has $n_{s_2\tau}$ connected components.
      \item For $3 \leq i \leq p+1$, $C_i^{[i-1]}$ is smooth of dimension two and the fiber over a closed point of $C_i^{[i-1]}$ is one dimensional, smooth, and connected.
      \item For $3 \leq i \leq p+1$, the exceptional locus $E^{[i]}$ is smooth and equidimensional of dimension three.
      \item For $3 \leq i \leq p+1$, $U_1^{[i]} \otimes \overline{\F}_p$ is equidimensional of dimension three.
     \end{compactenum}
     We proceed by induction, starting with the modification $\sA_1^{[4]} \rightarrow \sA_1^{[3]}$. Now (ii) and (iii) imply that the exceptional locus of $\sA_1^{[4]} \rightarrow \sA_1^{[3]}$ has the same number of connected components as the true center $\sC_4^{[3]}$, and that each such connected component is three dimensional and smooth. By (iv) each of these components is an irreducible component of $\sA_1^{[4]} \otimes \overline{\F}_p$, with all of the other irreducible components of $\sA_1^{[4]} \otimes \overline{\F}_p$ being given by the strict transform of the irreducible components of $\sA_1^{[3]} \otimes \overline{\F}_p$. Therefore by (i) there are $4 + n_{s_0s_2\tau} + n_{s_2\tau}$ irreducible components of $\sA_1^{[4]} \otimes \overline{\F}_p$.

     Assume the result is true for $i-1$ with $4 < i \leq p+1$. Once we show that $\sC_i^{[i-1]}$ has $n_{s_2\tau}$ connected components, the induction will follow using the same argument as in the above paragraph. Let $\sZ_i^{[i-1]}$ denote the strict transform of $\sZ_4'''$ with respect to the morphism $\sA_1^{[i-1]} \rightarrow \sA_1'''$, and likewise with $Z_i^{[i-1]}$ and the morphism $U_1^{[i-1]} \rightarrow U_1'''$. Inspection of the equations in the proof of Proposition \ref{Blowup} reveals that $C_i^{[i-1]} = Z_i^{[i-1]} \cap E^{[i-1]}$. With $\sZ_i^{[i-1]}$, $\sE^{[i-1]}$, and $\sC_i^{[i-1]}$ \'{e}tale locally corresponding to $Z_i^{[i-1]}$, $E^{[i-1]}$, and $C_i^{[i-1]}$ respectively, we therefore have $\sC_i^{[i-1]} = \sZ_i^{[i-1]} \cap \sE^{[i-1]}$.
     
     As before, (ii) and (iii) imply that $\sE^{[i-1]}$ has the same number of connected components as $\sC_{i-1}^{[i-2]}$, namely $n_{s_2\tau}$. Now (iii) also implies that if $\sZ_i^{[i-1]}$ meets a connected component $\sW \subset \sE^{[i-1]}$, then $\sZ_i^{[i-1]} \cap \sW$ is connected as otherwise the fiber above a point in $\sC_{i-1}^{[i-2]}$ would be disconnected. It therefore suffices to show that $\sZ_i^{[i-1]}$ meets each connected component of $\sE^{[i-1]}$. As $\sZ_i^{[i-1]} = \ST(\sZ_{i-1}^{[i-2]})$ and $\sA_1^{[i-1]} = \Bl_{\sZ_{i-1}^{[i-2]}}(\sA_1^{[i-2]})$, this is immediate.
    \end{proof}
    
    \begin{definition}\label{Graph definition}
     Let $p$ be an odd rational prime and $K^p \subset G(\A_f^p)$, which determines the numbers $n_{s_2\tau}$ and $n_{s_0s_2\tau}$ of $\sA_{0,K^p}$ (see the paragraph before Section \ref{sect Step 0: U_0, U_1, sA_0, and sA_1}). We then define the vertex-labeled graph $\Gamma_{K^p}$ as follows.
      \begin{compactenum}[(i)]
       \item Begin with $n_{s_2\tau}$ batons, each having $p-2$ vertices. Label the vertices $2p-3, 2p-4, \dots, p$ from head to tail.
       \begin{center}
        \includegraphics[scale=0.29]{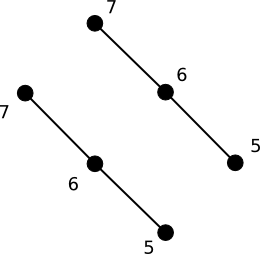}
        
        \captionof{figure}{Batons for $p = 5$ and $n_{s_2\tau} = 2$}
       \end{center}
 
       \item Add one vertex labeled $2p-2$ (top left) and edges between this vertex and the heads of the batons. Add two more vertices labeled $p-1$ (bottom left and top right) and connect these two vertices to every vertex in the batons, as well as the unique vertex labeled $2p-2$. Add $n_{s_0s_2\tau}$ vertices labeled $p-1$ (bottom right) and attach edges between these and the tails of the batons, as well as the other two vertices labeled $p-1$ from before.
       \begin{center}
        \includegraphics[scale=0.29]{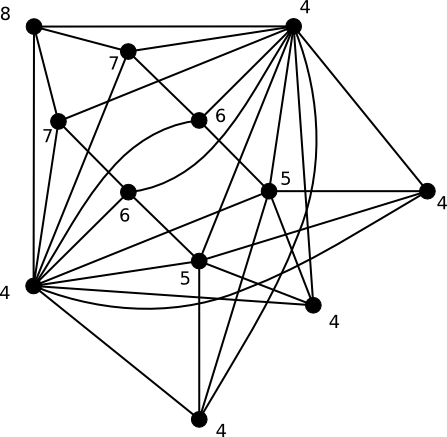}
        
        \captionof{figure}{Base for $p = 5$, $n_{s_2\tau} = 2$, and $n_{s_0s_2\tau} = 3$}
       \end{center}
 
       \item Add one vertex labeled $1$ and attach edges from this to every vertex constructed in the above two steps.
       \begin{center}
        \includegraphics[scale=0.3]{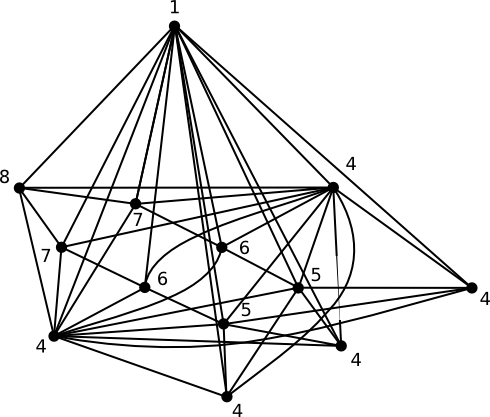}
        
        \captionof{figure}{Dual graph for $p = 5$, $n_{s_2\tau} = 2$, and $n_{s_0s_2\tau} = 3$}.
       \end{center}
     \end{compactenum}
    \end{definition}
    
    \begin{definition}
     We define the following subsets of the vertices of $\Gamma_{K^p}$.
     \begin{compactitem}
      \item The batons consist of the vertices given in step one above. They may be identified as the vertices with label in $[p,2p-3]$.
      \item The front consists of the vertices labeled $p-1$ on the bottom right of the diagram directly above. They may be identified as the vertices of label $p-1$ and edge degree $3+n_{s_2\tau}$ that share edges with precisely two vertices labeled 4.
      \item The sides consist of the vertices labeled $p-1$ which are not in the front.
     \end{compactitem}
    \end{definition}
    
    Recall that $\sA_0$ denotes a connected component of $\sA_0 \otimes_{\Z_p} \breve E$, and similarly with $\sA_1$.
    \begin{theorem}\label{thm description of resolution of sA_1}
     $\sA_1^{[p+1]} \rightarrow \sA_1$ is a resolution of singularities and the special fiber of $\sA_1^{[p+1]}$ is a nonreduced divisor with normal crossings. The special fiber has $4 + n_{s_0s_2\tau} + (p-2) \cdot n_{s_2\tau}$ irreducible components whose intersections are described by the vertex-labeled graph $\Gamma_{p,K^p}$ as follows.
     \begin{compactenum}[(i)]
      \item Each vertex represents an irreducible component. The label of the vertex is the multiplicity of the component.
      \item A full subgraph of $\Gamma_{K^p}$ which is a $k$-simplex indicates a $(k+1)$-fold intersection of the corresponding irreducible components, and conversely any $(k+1)$-fold intersection corresponds to a full subgraph which is a $k$-simplex. Such an intersection has dimension $3-k$ over $\Spec(\overline{\F}_p)$.
      \item Let $x^{[p-2]}$ be a closed point of $\sA_1^{[p-2]}$  and  $\set{e_1, \dots, e_t}$ be the multiset of the multiplicities of the irreducible components that $x^{[p-2]}$ lies on. Then there is an \'{e}tale neighborhood of $x^{[p-2]}$ of the form
       $$\Spec(\sO_{\breve E}[x_1, x_2, x_3, x_4]/(x_1^{e_1} \cdots x_t^{e_t} - p)).$$
      \item The following table gives the image of each irreducible component under the map $\sA_1^{[p+1]} \rightarrow \sA_0$.
      \begin{center}
       \begin{tabular}{c|L{7cm}}
        Component & Image\\
        \hline
        Front & Each irreducible component surjects onto a distinct connected component of $\overline{\sA_{s_0s_2\tau}}$\\
        \hline
        Sides & These two irreducible components surject onto the irreducible components $\sZ_{01}$ and $\sZ_{10}$ respectively.\\
        \hline
        Vertex labeled $2p-2$ & Surjects onto $\sZ_{11}$.\\
        \hline
        Vertex labeled $1$ & Surjects onto $\sZ_{00}$.\\
        \hline
        Batons & Fix a baton $B$. The irreducible components corresponding to the vertices in $B$ all surject onto the same connected component of $\overline{\sA_{s_2\tau}}$. This induces a bijection between the set of batons and the set of connected components of $\overline{\sA_{s_2\tau}}$.
       \end{tabular}
       \captionof{table}{Images of irreducible components}
      \end{center}
     \end{compactenum}
    \end{theorem}
 \section{A Resolution of $\sAGSp_{1,\bal}$ for $n=2$}
  \label{sect Iw_1^bal(p) resolution}
		We produce here a similar resolution as in the previous section in the case of the integral model of the Siegel modular varieties associated with $\Iw_1^\bal(p)$-level subgroup. However, we give only an outline of how to proceed. To simplify the notation, set $\sA_1 = \sAGSp_{1,\bal}$ and $U_1 = U_{1, \bal}^\GSp$. Begin by taking the normalization $\sA_1^\norm$ of $\sA_1$. As mentioned at the end of Section \ref{subsect Iw_1^bal(p)-level subgroup}, the normalization of $U_1$ is the spectrum of the ring $\Z_p[x,y,a,b,c,u_0,u_1,u_2,u_3]$ modulo the ideal generated by
		 $$xy - p, \qquad ax+by+abc, \qquad u_0u_3-u_1u_2,$$
		 $$u_0^{p-1}-x, \quad u_1^{p-1}-x-bc, \quad u_2^{p-1}-y-ac, \quad u_3^{p-1}-y.$$	
		We begin by naming the subschemes of $\sA_0$ that will transform to either be the center or true center of a blowup, using the notation as in previous section.
		\begin{center}
			\begin{tabular}{rcl|rcl}
				\multicolumn{3}{c|}{Local model $U_0$} & \multicolumn{3}{|c}{Integral model  $\sA_0$}\\
				\hline
				$Z_1$ &$=$& $Z(x,b)$ & $\sZ_1$ &$=$& $\sZ_{11}$\\
				$C_1$ &$=$& $Z(x,y,a,b)$ & $\sC_1$ &$=$& $\sZ_{00} \cap \sZ_{01} \cap  \sZ_{10} \cap \sZ_{11}$\\
				$Z_3$ &$=$& $Z(x,bc)$ & $\sZ_3$ &$=$& $\sZ_{11} \cup (\sZ_{01} \cap \sZ_{10})$\\
				$C_3$ &$=$& $Z(x,y,ac,bc)$ & $\sC_3$ &$=$& $\sZ_{01} \cap \sZ_{10}$\\
				$Z_4$ &$=$& $Z(x(x+bc), y(y+ac))$ & $\sZ_4$ &$=$& $\sZ_{01} \cup \sZ_{10}$\\
				$C_4$ &$=$& $Z(x,y,ac,bc)$ & $\sC_4$ &$=$& $\sZ_{01} \cap \sZ_{10}$\\
				$Z_5$ &$=$& $Z(x,b)$ & $\sZ_5$ &=& $\sZ_{11}$\\
				$C_5$ &$=$& $Z(x,y,b,c)$ & $\sC_5$ &=& $\overline{(\sZ_{01} \cap \sZ_{10} \cap \sZ_{11}) \setminus \sZ_{00}}$
			\end{tabular}
			\captionof{table}{Subschemes of $U_0$ and $\sA_0$}
		\end{center}
		
		The first three steps in constructing the resolution are precisely the same as before, namely $\sA_0' = \Bl_{\sZ_1}(\sA_0)$, $\sA_1'' = \sA_1^{\norm} \times_{\sA_0} \widetilde{\sA_0}$, and $\sA_1''' = \Bl_{\sZ_3''}(\sA_1'')$. Performing the corresponding constructions on the local models, $\sZ_3''$ corresponds to the subscheme $Z_3'' = Z(u_0, u_1) \subset U_1''$. Next $\sA_1^{[4]} = \Bl_{\sZ_4'''}(\sA_1'')$, where $\sZ_4'''$ \'{e}tale locally $\sZ_4'''$ corresponds to the subscheme $Z_4''' = Z(u_0,u_2) = Z(u_1,u_3)$.
		
		The scheme $U_1^{[4]}$ is covered by eight standard affine open charts. Four of these charts are regular with special fiber a nonreduced divisor with normal crossings. The blowups which are to follow induce isomorphisms over these four charts. The other four charts take the form the spectrum of
			$$Y = \breve{\Z}_p[x_1, x_2, a,b, u]/(u^{p-1}x_1^{p-1}x_2^{p-1} - p, u^{p-1}-ab).$$
		\'{E}tale locally, $\sZ_5^{[4]}$ corresponds to $Z(u,a)$ in each of these charts. Similar to Proposition \ref{Blowup}, we modify $Y$ by successively blowing up $Z(u,a)$ and its strict transforms in each resulting modification, doing this a total of $p-2$ times. The resulting resolution $\widetilde{\sA_1}$ is regular with special fiber a nonreduced divisor with normal crossings. By tracking the subschemes mentioned above through this construction, one may describe various properties of the irreducible components of the special fiber of $\widetilde{\sA_1}$ as in Theorem \ref{thm description of resolution of sA_1}. We merely note that all of the irreducible components of the special fiber of the resolution have multiplicity $p-1$.

%% file: LMG1_arxiv.bbl
\begin{thebibliography}{BBM}

\bibitem[BBM]{BBM}
P.~Berthelot, L.~Breen, and W.~Messing.
\newblock {\em Th\'eorie de {D}ieudonn\'e cristalline. {II}}, volume 930 of
  {\em Lecture Notes in Mathematics}.
\newblock Springer-Verlag, Berlin, 1982.

\bibitem[dJ1]{dJTalk}
A.~J. de~Jong.
\newblock Talk given in {W}uppertal.
\newblock 1991.

\bibitem[dJ2]{dJ}
A.~J. de~Jong.
\newblock The moduli spaces of principally polarized abelian varieties with
  {$\Gamma_0(p)$}-level structure.
\newblock {\em J. Algebraic Geom.}, 2(4):667--688, 1993.

\bibitem[DP]{DP}
P.~Deligne and G.~Pappas.
\newblock Singularit\'es des espaces de modules de {H}ilbert, en les
  caract\'eristiques divisant le discriminant.
\newblock {\em Compositio Math.}, 90(1):59--79, 1994.

\bibitem[EG]{Emerton-Gee}
M.~Emerton and T.~Gee.
\newblock {$p$}-adic {H}odge-theoretic properties of {\'{e}}tale cohomology
  with mod {$p$} coefficients, and the cohomology of {S}himura varieties.
\newblock {\em In preparation}.

\bibitem[EH]{EisenbudHarris}
D.~Eisenbud and J.~Harris.
\newblock {\em The geometry of schemes}, volume 197 of {\em Graduate Texts in
  Mathematics}.
\newblock Springer-Verlag, New York, 2000.

\bibitem[Fal]{Faltings}
G.~Faltings.
\newblock Toroidal resolutions for some matrix singularities.
\newblock In {\em Moduli of abelian varieties ({T}exel {I}sland, 1999)}, volume
  195 of {\em Progr. Math.}, pages 157--184. Birkh\"auser, Basel, 2001.

\bibitem[Gen]{Genestier}
A.~Genestier.
\newblock Un mod\`ele semi-stable de la vari\'et\'e de {S}iegel de genre 3 avec
  structures de niveau de type {$\Gamma_0(p)$}.
\newblock {\em Compositio Math.}, 123(3):303--328, 2000.

\bibitem[G{\"o}r1]{GortzGL}
U.~G{\"o}rtz.
\newblock On the flatness of models of certain {S}himura varieties of
  {PEL}-type.
\newblock {\em Math. Ann.}, 321(3):689--727, 2001.

\bibitem[G{\"o}r2]{GortzGSp}
U.~G{\"o}rtz.
\newblock On the flatness of local models for the symplectic group.
\newblock {\em Adv. Math.}, 176(1):89--115, 2003.

\bibitem[G{\"o}r3]{GortzComputing}
U.~G{\"o}rtz.
\newblock Computing the alternating trace of {F}robenius on the sheaves of
  nearby cycles on local models for {$\rm GL_4$} and {$\rm GL_5$}.
\newblock {\em J. Algebra}, 278(1):148--172, 2004.

\bibitem[GY1]{GYu}
U.~G{\"o}rtz and C.~Yu.
\newblock Supersingular {K}ottwitz-{R}apoport strata and {D}eligne-{L}usztig
  varieties.
\newblock {\em J. Inst. Math. Jussieu}, 9(2):357--390, 2010.

\bibitem[GY2]{GYu2}
U.~G{\"o}rtz and C.~Yu.
\newblock The supersingular locus in {S}iegel modular varieties with {I}wahori
  level structure.
\newblock {\em Math. Ann.}, 353(2):465--498, 2012.

\bibitem[Hai1]{HainesConnectedComponents}
T.~Haines.
\newblock On connected components of {S}himura varieties.
\newblock {\em Canad. J. Math.}, 54(2):352--395, 2002.

\bibitem[Hai2]{H}
T.~Haines.
\newblock Introduction to {S}himura varieties with bad reduction of parahoric
  type.
\newblock In {\em Harmonic analysis, the trace formula, and {S}himura
  varieties}, volume~4 of {\em Clay Math. Proc.}, pages 583--642. Amer. Math.
  Soc., Providence, RI, 2005.

\bibitem[HR]{HR}
T.~Haines and M.~Rapoport.
\newblock Shimura varieties with {$\Gamma_1(p)$}-level via {H}ecke algebra
  isomorphisms: The {D}rinfeld case.
\newblock {\em Ann. Scient. Ecole Norm. Sup.}, 45(4):719--785, 2012.

\bibitem[HS]{HS}
T.~Haines and B.~Stroh.
\newblock In preparation.

\bibitem[HT]{HT}
M.~Harris and R.~Taylor.
\newblock Regular models of certain {S}himura varieties.
\newblock {\em Asian J. Math.}, 6(1):61--94, 2002.

\bibitem[KM]{KatzMazur}
N.~Katz and B.~Mazur.
\newblock {\em Arithmetic moduli of elliptic curves}, volume 108 of {\em Annals
  of Mathematics Studies}.
\newblock Princeton University Press, Princeton, NJ, 1985.

\bibitem[Kot]{Ko92}
R.~Kottwitz.
\newblock Points on some {S}himura varieties over finite fields.
\newblock {\em J. Amer. Math. Soc.}, 5(2):373--444, 1992.

\bibitem[KR]{KR}
R.~Kottwitz and M.~Rapoport.
\newblock Minuscule alcoves for {${\rm GL}_n$} and {$G{\rm Sp}_{2n}$}.
\newblock {\em Manuscripta Math.}, 102(4):403--428, 2000.

\bibitem[Mat]{Matsumura}
H.~Matsumura.
\newblock {\em Commutative algebra}, volume~56 of {\em Mathematics Lecture Note
  Series}.
\newblock Benjamin/Cummings Publishing Co., Inc., Reading, Mass., second
  edition, 1980.

\bibitem[OT]{OT}
F.~Oort and J.~Tate.
\newblock Group schemes of prime order.
\newblock {\em Ann. Sci. \'Ecole Norm. Sup. (4)}, 3:1--21, 1970.

\bibitem[Pap1]{Pappas95}
G.~Pappas.
\newblock Arithmetic models for {H}ilbert modular varieties.
\newblock {\em Compositio Math.}, 98(1):43--76, 1995.

\bibitem[Pap2]{P}
G.~Pappas.
\newblock On the arithmetic moduli schemes of {PEL} {S}himura varieties.
\newblock {\em J. Algebraic Geom.}, 9(3):577--605, 2000.

\bibitem[PR1]{PR1}
G.~Pappas and M.~Rapoport.
\newblock Local models in the ramified case. {I}. {T}he {EL}-case.
\newblock {\em J. Algebraic Geom.}, 12(1):107--145, 2003.

\bibitem[PR2]{PR2}
G.~Pappas and M.~Rapoport.
\newblock Local models in the ramified case. {II}. {S}plitting models.
\newblock {\em Duke Math. J.}, 127(2):193--250, 2005.

\bibitem[PR3]{PR3}
G.~Pappas and M.~Rapoport.
\newblock Local models in the ramified case. {III}. {U}nitary groups.
\newblock {\em J. Inst. Math. Jussieu}, 8(3):507--564, 2009.

\bibitem[PRS]{LocalModelSurvey}
G.~Pappas, M.~Rapoport, and B.~Smithling.
\newblock Local models of {S}himura varieties, {I}. {G}eometry and
  combinatorics.
\newblock In {\em Handbook of moduli. {V}ol. {III}}, volume~26 of {\em Adv.
  Lect. Math. (ALM)}, pages 135--217. Int. Press, Somerville, MA, 2013.

\bibitem[PZ]{PZ}
G.~Pappas and X.~Zhu.
\newblock Local models of {S}himura varieties and a conjecture of {K}ottwitz.
\newblock {\em Invent. Math.}, 194(1):147--254, 2013.

\bibitem[RZ]{RZ}
M.~Rapoport and Th. Zink.
\newblock {\em Period spaces for {$p$}-divisible groups}, volume 141 of {\em
  Annals of Mathematics Studies}.
\newblock Princeton University Press, Princeton, NJ, 1996.

\bibitem[Sch]{ScholzeLanglands-KottwitzApproach}
P.~Scholze.
\newblock The {L}anglands-{K}ottwitz approach for some simple {S}himura
  varieties.
\newblock {\em Invent. Math.}, 192(3):627--661, 2013.

\bibitem[Yu]{YuSiegel}
C.~Yu.
\newblock Irreducibility and {$p$}-adic monodromies on the {S}iegel moduli
  spaces.
\newblock {\em Adv. Math.}, 218(4):1253--1285, 2008.

\end{thebibliography}
